\documentclass[11pt]{amsart}
\usepackage{amsmath,amssymb,amsthm,hyperref}
\usepackage{graphicx,graphics}
\usepackage{tikz}  
\usepackage{epstopdf}
\usepackage{arcs}
\usepackage{psfrag}          
\usepackage{buczomca}

\usepackage{version} 
\numberwithin{equation}{section}
\usepackage{enumerate}
\usepackage{datetime}
\usepackage[bf]{caption}     
\usepackage{bm}  

 \newtheorem{theorem}{Theorem}[section]
 \newtheorem{corollary}[theorem]{Corollary}
 
 \newtheorem{lemma}[theorem]{Lemma}
 \newtheorem{proposition}[theorem]{Proposition}
 
 \theoremstyle{definition}
 \newtheorem{definition}[theorem]{Definition}
 \newtheorem{remark}[theorem]{Remark}
 \newtheorem{question}[theorem]{Question}
 
   \newtheorem{conjecture}[theorem]{Conjecture}
 
\usepackage{color}
\definecolor{aquam}{rgb}{0.5,1.0,1.0}
\definecolor{bbrown}{rgb}{0.75,0.38,0.15}
\definecolor{Cyan}{rgb}{0,0.6,0.6}
\definecolor{Darkblue}{rgb}{0,0,1}
\definecolor{Dodgerblue2}{rgb}{0,0.5,1}
\definecolor{Green}{rgb}{0,0.5,0.1}
\definecolor{Kahki}{rgb}{1,1,0.5}
\definecolor{Magenta}{rgb}{1,0,1}
\definecolor{bMagenta}{rgb}{1,.6,1}
\definecolor{Orange}{rgb}{0.8,0.3,0}
\definecolor{dOrchid}{rgb}{0.7,0.2,0.4}
\definecolor{Orchid}{rgb}{1,0.5,1}
\definecolor{Purple}{rgb}{0.65,0.07,0.85}
\definecolor{Royalblue}{rgb}{0.6,0.85,0.87}
\definecolor{Tan}{rgb}{0.54,0.42,0.23}
\definecolor{bTan}{rgb}{0.94,0.82,0.63}
\definecolor{Turquoise}{rgb}{0,0.85,0.87}
\definecolor{Yellow}{rgb}{1,1,0}
\definecolor{bYellow}{rgb}{1,1,0.6}
\definecolor{bRed}{rgb}{1,0.7,0.7}
\definecolor{boxcolb}{rgb}{0.87,0.77,0.75}
\definecolor{boxcol}{rgb}{0.6,0.85,0.87}
\definecolor{boxcolgreen}{rgb}{0.64,0.93,0.79}
\definecolor{boxcolaa}{rgb}{.75,.99,.70}
\definecolor{boxcolbb}{rgb}{0.39,0.40,0.56}
\definecolor{boxcolcc}{rgb}{1,0.81,0.65}
\definecolor{yy}{rgb}{0.43,0.21,.18}
\definecolor{gA}{gray}{0.5}
\definecolor{gB}{gray}{0.8}
\definecolor{gC}{gray}{0.9}

\newcommand\sk{\smallskip}
\newcommand\mk{\medskip}

 \newcommand \cal{\mathcal}     
 \newcommand\zu{[0,1]}
\newcommand{\N}{\ensuremath{\mathbb N}} 
\newcommand{\R}{\ensuremath{\mathbb R}} 
\newcommand{\Z}{\ensuremath{\mathbb Z}} 

\newcommand{\ep}{\varepsilon}
\newcommand\du{\overline{d}}
\newcommand\dl{\underline{d}}

\newcommand\wq{\widetilde{Q}}

\renewcommand{\ggg}{\gamma}

\usepackage{xcolor}
\hypersetup{
  colorlinks   = true, 
  urlcolor     = blue, 
  linkcolor    = blue, 
  citecolor   = red 
}

\title[Measures and Annuli]{Measures, annuli and dimensions}
\author[Z. Buczolich]{Zolt\'an Buczolich$^*$}
\address{Department of Analysis, ELTE E\"otv\"os Lor\'and\\
University, P\'azm\'any P\'eter S\'et\'any 1/c, 1117 Budapest, Hungary}
\email{zoltan.buczolich@ttk.elte.hu}
\urladdr{http://buczo.web.elte.hu, ORCID Id: 0000-0001-5481-8797}
\author[S. Seuret]{St\'ephane Seuret$^\text{\textdagger}$ }
\address{St\'ephane Seuret, Universit\'e Paris-Est, LAMA (UMR 8040),  UPEMLV, UPEC, CNRS, F-94010, Cr\'eteil, France}
\email{seuret@u-pec.fr }

\thanks{\scriptsize $^*$
The project leading to this application has received funding from the European Research Council (ERC) under the European Union’s Horizon 2020 research and innovation programme (grant agreement No. 741420).
This author was supported by the Hungarian National Research, Development and Innovation Office--NKFIH, Grant 124003 and  at the time of completion of this paper was holding a visiting researcher position 
at the Rényi Institute.
}

\thanks{\scriptsize $^\text{\textdagger}$ This author was supported by Grant ANR-16-CE33-0020 MULTIFRACS.
}

\date{\today}
\begin{document}
\maketitle

\medskip

\begin{abstract}

Given a Radon probability measure $\mu$ supported in $\mathbb{R}^d$, we are interested in those points  $x$ around which the measure is concentrated infinitely many times on  thin annuli  centered  at $x$.
Depending on the lower and upper dimension of $\mu$,  the metric used in the space
and   the thinness  of the annuli, 
we obtain results and examples when such points are of $\mu $-measure $0$ or of $\mu$-measure $1$.

The measure concentration we study is related  to "bad points"  for the Poincaré recurrence theorem
and to 
the first return
times to shrinking balls under iteration generated by a weakly Markov  dynamical system.

The study of thin annuli and spherical averages is also important in many dimension-related problems, including Kakeya-type problems and
Falconer’s distance set conjecture. 
\end{abstract}

\setcounter{tocdepth}{2}
\tableofcontents

\section{Introduction and  main results}\label{introduction}

In the following,  $\mu$ is
a Radon probability measure  supported in $\mathbb{R}^d$ and
$\dim_{H}$ denotes the Hausdorff dimension. We denote by
$B(x,r)$ the closed ball $\{y\in \R^d: \|x-y\|\leq r\}$, which obviously depends on the norm $\|\cdot\|$ chosen on $\R^d$.   We consider norms which are equivalent with
the most common Euclidean one,
$||.||_{2}$.
\begin{definition}
For every $x\in \R^d$, $0<r<1$ and $\delta \geq 1$,   define the annulus
\begin{equation}
\label{defA}
A(x,r,\delta) = B (x,r)\setminus B (x,r-r^\delta).
\end{equation}

We say that 
$$\text{$P_{\mu }(x,r,\delta,\eta)$ holds  when } \mu\big(A (x,r,\delta)\big) \geq \eta  \cdot \mu\big(B(x,r)\big).$$
 
Finally, we set 
$$E_{\mu}(\delta,\eta) = \{x\in \R^d: P_{\mu}(x,r_n,\delta,\eta) \mbox{ holds for a sequence $(r_n)_{n\geq 1}\to 0$}\}.$$
\end{definition}

Intuitively,  around points belonging to $E_{\mu}(\delta,\eta) $, the measure $\mu$ concentrates a substantial part of its local mass on a very thin annulus (since $r^\delta<\!\!\!<r$). The larger $\delta$, the thinner the annulus: $E_{\mu}(\delta',\eta)  \subset E_{\mu}(\delta,\eta) $ when $\delta' \geq \delta$.  Our goal is to investigate the size of the sets $E_{\mu}(\delta,\eta) $.

In this paper, we only consider diffuse measures, i.e. without any Dirac mass: $\mu(B(x,0)) = 0$, for every $x\in \R^d$.
In this case,   $\mu(E_\mu(1,\eta)) = 1$ for every $\eta\in [0,1]$. The question we investigate hereafter concerns the size of  $E_\mu(\delta,\eta)$ for $\delta>1$, and it appears that the answer depends on the measure $\mu$,   the thinness $\delta$ and the norm used to define the annuli, in a subtle manner.


\medskip

The sets $E_\mu(\delta,\eta)$ appear in various places. For instance, in \cite{CHAZ}, it is proved that if $\mu$ is   the Sinai-Ruelle-Bowen measure   associated with  a non-uniformly hyperbolic dynamical system $(X,T,\mu)$,  then  the elements  of $E_\mu(\delta,\eta)$ are "bad points" 
for the Poincar\'e recurrence theorem, in the sense that given $r>0$, when $x\in E_\mu(\delta,\eta)$, the iterates  $T^j x$ of $x$  come back inside $B(x,r)$ not as often as expected.  

More recently, Pawelec,   Urba\'nski,  and Zdunik \cite{Urbanski-annuli} investigated the first return
times to shrinking balls under iteration generated by a weakly Markov  dynamical systems, and had to deal with what they call the {\em Thin Annuli Property}. 
This property has several versions in  \cite{Urbanski-annuli}, and  is very similar to  belonging to the complementary set  of our  sets $E_{\mu}(\delta,\eta) $, except that the exponent $\delta$  in $P_{\mu}(x,r,\delta,\eta) $  depends on $r$, and may tend   to infinity when $r$ tends to $0$. In weaker versions of the Thin Annuli Property there are also restrictions on the range of radii.
The conditions we impose to the elements of  $E_{\mu}(\delta,\eta) $  are   stronger, 
that is, they imply that the so-called   {\em Full Thin Annuli Property} of  \cite{Urbanski-annuli}  holds.   In the same paper, the authors prove (Theorem C) that  every finite Borel measure $\mu$ in a Euclidean space $\R^d$, satisfies the
Thick Thin Annuli Property (this means that for arbitrary measures the range of radii for which the Thin Annuli Property holds is more limited). Our theorems below state that in many situations (for instance, for all  measures $\mu$ with large lower dimensions), Theorem C can be improved.

Let us also mention   that Theorem D of \cite{Urbanski-annuli} shows  that  certain measures coming from  conformal geometrically irreducible Iterated Function Systems satisfy   the Full Thin Annuli Property. 
We will come back to this later in the introduction. 

Similar questions appear also when studying orbit distribution of various groups acting on $\R^2$ (Theorem 3.2 of \cite{Schap1}). See also   \cite{HaydnWa,Schap2} for other occurrences  of such questions. Connections  with other works are also made later in the introduction.


\medskip

We start by proving that, regardless of the norm in $\R^d$,    measures with large lower dimension   do not charge annuli at small scales if the exponent $\delta$ defining the annuli is sufficiently large,
where "sufficiently large" depends on the lower and upper dimensions of $\mmm$, whose definitions are recalled now.

\begin{definition}
 Let $\mu$ be a Radon probability measure on $\R^d$.
 
 The lower and upper dimensions of $\mu$  are defined as
\begin{align*}\underline{\dim}(\mu)=  \sup \{\alpha\geq 0: & \ \mbox{for $\mu$-a.e $x$, }\\ &\mbox{$\exists \,r_x>0$,    $\forall \, 0<r<r_x$, $\mu(B(x,r)) \leq r^\alpha$}\}
\end{align*}
and
\begin{align*}
\overline{\dim}(\mu)= \inf \{\beta\geq 0:& \ \mbox{for $\mu$-a.e $x$, } \\ & \mbox{$\exists \,r_x>0$,    $\forall \, 0<r<r_x$,  $\mu(B(x,r)) \geq r^\beta$}\}.
\end{align*}
\end{definition}

Our first result is the following.
\begin{theorem}
\label{mainth1}
Let $\mu$ be a probability measure on $\R^d$ such that $\underline{\dim}(\mu) > d-1$. 

For every $\delta> \frac{\overline{\dim}(\mu) -(d-1)}{\underline{\dim}(\mu) -(d-1)}$ and $\eta \in (0,1]$, one has
$\mu\big(E_\mu(\delta,\eta)\big) =0$.
\end{theorem}

Hence, for mono-dimensional measures $\mu$ satisfying $\overline{\dim}(\mu) = \underline{\dim}(\mu)>d-1$, $\mu\big(E_\mu(\delta,\eta)\big) =0$ for every $\delta>1$. Observe that Theorem \ref{mainth1} holds true regardless of the underlying metric used to define $A (x,r,\delta) $.
Also, in dimension $d=1$, Theorem \ref{mainth1} is simpler and rewrites as follows:  $\mu\big(E_\mu(\delta,\eta)\big) =0$ for every  $\delta> \frac{\overline{\dim}(\mu)}{\underline{\dim}(\mu) }$ and $\eta \in (0,1]$,

\mk

Next theorem
 shows that Theorem \ref{mainth1} is optimal if   the  $||.||_{\oo}$ metric is used.

\begin{theorem}
\label{mainth2} 
 Suppose that the metric 
 generated by the norm
 $||.||_{\oo}=\max \{ |x_{i}|:\, i=1,...,d \}$ is used to define the annuli $A(x,r,\delta)$  in \eqref{defA}.
  
For every $d-1<    \underline{d}  <  \overline{d}  < d$ and    every $  \eta\in (0,1) $, there exists a probability measure $\mu$ on $\mathbb{R}^{d}$ such that $\underline{\dim}(\mu) = \underline{d}$, $\overline{\dim}(\mu) =\overline{d}$  and 
\begin{equation}
\label{eqmu}
\mu\left(E_\mu\left(  \frac{\overline{d}-(d-1)}{\underline{d}-(d-1)  },\eta\right)\right ) =1. 
\end{equation}

\end{theorem}

 \begin{remark}
In Theorem \ref{mainth2} the case $\overline{d} =\underline{d} $ is trivial
since as noticed above, $\mmm(E_{\mmm}(1,\hhh))=1$ is always true for any non-atomic measure. 
 \end{remark}

\mk

Still in the $||.||_{\oo}$  case,   we further  investigate what happens  for measures of   lower dimension  less than $d-1$. The quite surprising result is that for such measures $\mu$,    the worse scenario may always happen, in the sense that it is possible that $\mu$ charges only points around which the mass is infinitely often concentrated on small annuli. 

\begin{theorem}
\label{mainth3}

 Suppose that $d\geq 2$ and that  the  metric $||.||_{\infty}$ is used.
For every $   \underline{d} \leq d-1$,   $ \underline{d} \leq  \overline{d}\leq d$,   every   $  \eta\in (0,1) $  and every $\delta> 1$, 
 there exists a probability measure $\mu$ on $\mathbb{R} ^{d} $ such that $\underline{\dim}(\mu) = \underline{d}$, $\overline{\dim}(\mu) =\overline{d}$  and 
\begin{equation}
\label{conclusion3}
 \mu\big(E_\mu(\delta ,\eta)\big ) =1.
 \end{equation}
 \end{theorem}


Although we do not explicitly state it, an adaptation of their proofs show  that  Theorems \ref{mainth2}  and \ref{mainth3}
 remain true when the frontiers of the annuli   $ A (x,r,\delta) $ in the given metric are  finite unions of convex parts of hyperplanes, for instance  in the case
$||.||=||.||_{1}=\sum_{i=1}^{d}|x_{i}|$.

\medskip

While the proof of Theorem \ref{mainth1} deals with all measures satisfying its assumptions, the proofs of Theorems \ref{mainth2} and \ref{mainth3} are constructive (they are both  based on  the same arguments): we explicitly build measures such that \eqref{eqmu} or \eqref{conclusion3} are true.

\medskip

Coming back to Theorem \ref{mainth1}, it is striking that when the Euclidean metric is used,  the uniform bound for  $\delta$ can  be improved, in the sense that $\mu(E_\mu(\delta,\eta))=0$ even for $\delta $ smaller than  $ \frac{\overline{d}-(d-1)}{\underline{d}-(d-1)  }$.  Next theorem illustrates this fact   when $d=2$, even when $\underline{d}<d-1=1$.

\begin{theorem}\label{mainth4}
 Suppose that $d=2$ and that the Euclidean metric $||.||_{2}$ is used.
 
 Let  $\overline{d}  \in [0.89,  2]$ and    $\underline{d}\in [0.89,  \overline{d}]$.  Suppose that $\mmm$ is a Radon probability measure  such that ${\underline{\mathrm{dim}}\, } \mmm=\underline{d}$ and ${\overline{\mathrm {dim}}\, } \mmm=\overline{d} $. Then 
 $\mmm(E_\mu(30,\eta))=0$
 for any $\hhh\in (0,1).$

\end{theorem}

In the above Theorem \ref{mainth4}, taking $\underline{d}=1.01$ and $\overline{d}=1.99$, one sees that 
$$\frac{\overline{d} -(d-1)}{\underline{d} -(d-1)}=\frac{1.99-1}{1.01-1}=99.$$
By Theorem \ref{mainth2}, one might expect the existence of a probability measure $\mmm$ for which $\mmm(E_\mu(99,\eta))=1$.
Since $E_\mu(30,\eta)\supset E_\mu(99,\eta)$ the result of Theorem \ref{mainth4} goes well beyond
the bound in \eqref{eqmu} and shows that in the Euclidean metric, annuli are sufficiently
``independent/decorrelated" so that  Theorem \ref{mainth1} can be sharpened significantly -observe that Theorem \ref{mainth4} holds for all measures satisfying its assumptions.

The heuristic intuition explaining the difference between  Theorems \ref{mainth3}  and \ref{mainth4} is that when an annulus  with a cubic  shape centered at a point $x$ is translated by a very small distance,
 a large part of the translated  annulus is still contained in a cubic annulus centered at $x$ with comparable sidelength. But this does not hold true anymore for annuli with spherical shape. 
 More generally, it is standard that dealing with the Euclidean norm is often more complicated than with polyhedral norms in many dimensional problems (we   come back to this below). 
Our key tool to prove Theorem \ref{mainth4} is Lemma \ref{lemmcor}, which is an estimate of the size of intersecting annuli. This type of estimates were considered by many authors 
see, for example \cite{Boumax}, \cite{Marcirc} or especially Lemma 3.1 of \cite{WolffKakeyasur}. The order we obtain in Lemma \ref{lemmcor} is slightly better than
the ones available in the literature, and optimal as we remark in Section 
\ref{proof4}.

Also, it is striking that Theorem \ref{mainth4} deals also with lower and upper dimensions for $\mu$ that are less than $1=d-1$, emphasizing the difference between the $||.||_{\infty}$ metric  (and Theorem \ref{mainth2}) and the Euclidean metric.
 
 \medskip
 
The values 0.89 and 30 we obtain are not optimal, and obtaining  exact bounds in Theorem \ref{mainth4} for the Euclidean metric in dimension $d$  seems to be a challenging and interesting open problem.

 \begin{question}
 Suppose that the Euclidean metric is used in $\R^d$. For every $0 < \underline{d}\leq \overline{d}\leq d$, find the best  $1\leq \delta=\delta(\overline{d}, \underline{d}) $ such that for every probability measure $\mu$ supported inside $\zu^d$, for every  $\delta'>\delta$,  for every $\eta\in (0,1)$, $\mu\big(E_\mu(\delta',\eta)\big) =0$.
 \end{question}

Given our previous results, it is natural to conjecture the following:
\begin{conjecture}\label{*conject}
When $\underline{d} > d-1$, the optimal $\delta(\overline{d}, \underline{d}) $ is such that $\delta(\overline{d}, \underline{d}) <\frac{\overline{d}-(d-1)}{\underline{d}-(d-1)  }$.
\end{conjecture}

{\bf Application to dynamical systems.} 
Suppose that $(T, X, \mu , \varrho )$
is a metric measure preserving dynamical system, that is $(X, \varrho )$ is a
metric space and $T : X \to  X$ is a Borel measurable map preserving a Borel probability
measure $\mu$  on $X.$ Given a ball $B(x,r)$ and $y\in X$,
$$\tau_{B(x,r)} (y) := \min\{n  \geq  1 : T^n(y) \in  B(x,r)\},$$
 is the first entry time of $y$ to $B(x,r)$. When $y\in B(x,r)$,  it is called the first return
time of $y$ to $B(x,r)$.

The entry and return times $\tau_{B(x,r)} (y) $ are studied in \cite{Urbanski-annuli}, for Weakly Markov systems $(T, X, \mu , \varrho )$  (we refer to \cite{Urbanski-annuli} for precise definitions). From  Theorems \ref{mainth1} and \ref{mainth4}, 
it follows that for certain measures the {\it Full Thin Annuli Property}  from
\cite{Urbanski-annuli} is satisfied. 
This way, based on Theorem B of \cite{Urbanski-annuli} one can state the following theorem.

\begin{theorem}\label{*thurbi}
 Let $(T, X, \mu , \varrho )$ be a Weakly Markov system, with $X\sse \R^{d}$. 
If one of the following two conditions is satisfied:
\begin{enumerate}
\item  $\underline{\dim}(\mu) > d-1$ and the metric used is any of the equivalent metrics used in $\R^{d}$;
\item $d=2$, the metric is the Euclidean and $\underline{\dim}(\mu) \geq  0.89$;
\end{enumerate}
then
 the distributions of the normalized first entry time and first return time converge
to the exponential one law, that is
\begin{equation}\label{*urb11}
\lim_{r\to 0}\sup_{t>0}\Big |  \mu\Big( \Big \{   y\in X:\tau_{B(x,r)}(y)>\frac{t}{\mu(B(x,r))} \Big  \} \Big )-e^{-t}   \Big |=0
\end{equation}
and
\begin{equation}\label{*urb12}
\lim_{r\to 0}\sup_{t>0}\Big |  \frac{1}{\mu(B(x,r))}\mu\Big( \Big \{   y\in B(x,r):\tau_{B(x,r)}(y)>\frac{t}{\mu(B(x,r))} \Big  \} \Big )-e^{-t}   \Big |=0.
\end{equation}
\end{theorem}

In \cite{Urbanski-annuli}, the conclusions are true for every Weakly Markov system $(T, X, \mu , \varrho )$, with the restriction that  the limits   in  \eqref{*urb11} and \eqref{*urb12}  are taken only on a subsequence of radii (more precisely, the limit is $\lim_{r\to 0, \, r\in R_{x}}$,
where $R_{x}$ is a $\beta$-thick class of radii - for the exact details see \cite{Urbanski-annuli}).  Our main improvement is to show that the limit holds true {\em for every measure} provided that one of the two conditions is satisfied.  

\medskip

Let us mention that the problems concerning intersecting annuli in
Section \ref{proof4} are reminiscent to questions arising when studying for instance Falconer's distance set conjecture  (for which recently many striking results were recently obtained \cite{Fadist,falc-inter-distance,Iosevitch-Laba,GuthFad,MattilaFou}) and the (circular) Kakeya problem \cite{Keleti-union-circles,Keleti-squares-center}. 
Distribution of measures on annuli plays an important role in these problems as well, mainly through the study of  cubic or spherical averages, and it is a standard issue  that the choice of the norm influences the results.  In Section \ref{*secdist} we   discuss this in more detail, in particular, based on standard arguments \cite{Fadist,MattilaFou,Taomeas} in Fourier and potential theory, the following proposition is proved.
 
 \begin{proposition}\label{*anonym}
Let $t > 1/2,$ and let $\mu$ be a finite $t$-regular measure on $\R^2$, i.e.  a Radon measure  satisfying 
\begin{equation}\label{*treg}
c_{t} r^{t}\leq \mmm(B(x,r))\leq C_{t}r^{t}, \ \ax x\in\spt(\mmm)\text{ and }
0<r<\diam(\spt \mmm).
\end{equation}
Assume that $\mmm$ has compact support.

Then,
for $\delta=4$, for every $\eta > 0$
\begin{equation}\label{*eqanonym}
\lim_{r\to 0} \mu(\{x\in \R^2 : P_{\mu }(x,r,4,\eta) \mbox{  holds} )\})=0.
\end{equation}
\end{proposition}

Unfortunately, this convergence in measure, or similar arguments, do not help improving our bounds (in fact, as far as we checked they do not even yield Theorems \ref{mainth1} to \ref{mainth4}). But it is quite interesting that similar issues arise in both problems.

The paper is organized as follows.

In Section \ref{proof1}, Theorem \ref{mainth1} is proved. The proof is  natural and  quite short, based on  the Radon measure version of Lebesgue's density theorem,
Corollary \ref{*Ma214}.

In Section \ref{proof2}, we explicitly build a measure $\mu$ such that $\mu(E_\mu(\delta,\eta))=1$, for any choice of $\delta$ and $\eta$. The construction is based on two subdivision schemes {\bf A} and {\bf B} that allow to   spread the mass of a cube on its boundaries in a controlled manner.

In Section \ref{proof3}, we show that the construction of Section \ref{proof2} can be adapted to prove Theorem \ref{mainth3}.

In Section \ref{proof4}, the Euclidean case (and Theorem \ref{mainth4}) is studied.

Finally,  in Section \ref{*secdist} we prove Proposition \ref{*anonym} and explain why such arguments, though interesting,  are for the moment not strong enough to reach Theorems \ref{mainth1} to \ref{mainth4}.

\section{Proof of Theorem \ref{mainth1}}
\label{proof1}


Before starting the proof we recall with slight change of notation part (1) of
2.14 Corollary from \cite{MATTILA}.

\begin{corollary}\label{*Ma214}
Suppose that $\mu$ is a Radon measure on $\R^{n}$ and $E\sse \R^n$ is $\mu$ measurable. Then the limit
$$\lim_{r\searrow 0}\frac{\mu(E\cap B(x,r))}{\mu(B(x,r))}$$
exists and equals $1$ for $\mu$-almost all $x\in E$ and equals $0$
for $\mu$-almost all $x\in \R^{n}\sm E$.
\end{corollary}

\begin{proof}[Proof of Theorem \ref{mainth1}]
Fix $\delta>(\du-(d-1))/(\dl-(d-1))$ and  $\ep>0$ so small that 
\begin{equation}\label{*es}
(\dl-\ep+(d-1) )\delta > \du+\ep+(d-1) .
\end{equation}

Observe that  there exists a constant $C_d>0$, depending on $d$ and the chosen norm only, such that for any ball $B(x,r)$, the associated annulus   $A(x,r,\delta)$   can be covered by at most $C_d  r^{(d-1)(1-\delta)} $ smaller balls  $B$ of radius $r ^\delta$.

Also, choose $0<\eta<1$, and consider $ E_\mu(\delta,\eta)$.

Proceeding  towards   a   contradiction,   suppose   that $\mu(E_\mu(\delta,\eta))>0$.

Consider for every $r,\ep>0$ the set 
  \begin{equation}\label{*Bepr*}
D_{\ep,r} = \{x \in \R^d:   \forall \, 0<s\leq r, \  s^{\du+\ep}\leq  \mu(B(x,s)) \leq s^{\dl-\ep}\}.
 \end{equation}

By definition, for every $\ep>0$,  the set  $  E_\mu(\delta,\eta) \cap   \bigcup_{p\geq 1} D_{\ep,1/p}   $ has full $\mu$-measure in $E_\mu(\delta,\eta)$. This holds especially for $\eee$ fixed in \eqref{*es}.

We   put   $E= E_\mu(\delta,\eta) \cap  D_{\ep, 1/p} $ with 
  a choice of    a    sufficiently large  $p$ such that   $\mu(  E  ) \geq   \mu(E_\mu(\delta,\eta))/2>0$.

Finally, we choose $0<\gamma <1/p$ so small  that   for every $0\leq r\leq \gamma$,
\begin{equation}
\label{eq3}
r^{(\dl-\ep-(d-1))   \delta} \leq    \frac{\eta  2^{  \ep -\dl  }   }{2C_d  } \,  r^{ \du+\ep-(d-1)} . 
\end{equation}


For every $x \in E$, there exists $r_x>0$, such that   
\begin{equation}\label{*Best}
\text{  $r_x< \gamma/2$,  \ \ \ 
$r_x^{\du+\ep}\leq  \mu(B(x,r_x)) \leq r_x^{\dl-\ep}$}
\end{equation}
$$\text{and  \ \ $P_{\mu}(x,r_x,\delta,\eta)$ holds.}$$

By using Corollary \ref{*Ma214} we can also assume that for $\mu$-almost all $x\in E$, we have chosen $r_x$  so small that  
$${\mu(E\cap B(x,r_x))}> (1-\eta/10) {\mu(B(x,r_x))}$$
and hence
\begin{equation}\label{*densest}
\mmm(B(x,r_x)\sm E)< (\eta/10) {\mu(B(x,r_x))}.
\end{equation}

Let us write $B_x=B(x,r_x)$ and $A_{x}=A(x,r_x,\ddd)$ for every $x\in E$. Such a ball satisfies by \eqref{*Best}
\begin{equation}
\label{eq2}
r_x^{\du+\ep}\leq  \mu(B_x) \leq r_x^{\dl-\ep}
\end{equation}
together with  $P_{\mu}(x,r_x,\delta,\eta)$.

Since $\mmm(E)>0$, it is possible to select an $\mathbf{x}\in E$ for which \eqref{*densest} holds.

  Since   $P_{\mu}(\mathbf{x},r_\mathbf{x},\delta,\eta)$   is   satisfied, we have
\begin{equation}
\label{ineg2}
 \eta \cdot \mu( B_\mathbf{x} )  \leq \mu  ( A_\mathbf{x} ) .
 \end{equation}
 
In addition, by definition of $C_d$,  $A_\mathbf{x}$ is covered by at most $C_d (r_\mathbf{x})^{(d-1)(1-\delta)} $ 
many balls  $B$ of radius $r_{\mathbf{x} }^\delta$. For each of these balls $B$, either $\mu (E\cap B)=0$, or $\mu(E\cap B) >0  $ and in this case $E\cap B\subset B(y, 2 r_{\mathbf{x} }^\delta )$ for some $y\in E$. 
Observe that by $\delta\geq 1$ and \eqref{*Best}, we have $2 r_{\mathbf{x} }^\delta<2 r_{\mathbf{x} }<\ggg<1/p$.
  By \eqref{*Bepr*}   
\begin{align*}
\mu(E\cap B)  &  \leq \mu(B(y, 2 r_{\mathbf{x} }^\delta  )) \leq (2 r_{\mathbf{x} }^\delta )^{\dl-\ep}       .
\end{align*}
 Hence, summing over   the (at most $C_d  (r_\mathbf{x})^{(d-1)(1-\delta)}  $) balls that cover $A_\mathbf{x}$, we get  by \eqref{eq3} and \eqref{eq2}  that
\begin{align*}
  \mu(A_\mathbf{x}\cap E) & \leq C_d (r_\mathbf{x})^{(d-1)(1-\delta)} (2 r_{\mathbf{x} }^\delta )^{\dl-\ep}     \leq   \frac{\eta}{2}  r_\mathbf{x}^{\du+\ep}   \leq   \frac{\eta}{2} \mu(B_\mathbf{x}). 
\end{align*}
Since 
by \eqref{*densest} and \eqref{ineg2}
$$\mu(A_\mathbf{x}\cap E)\geq \mu(A_\mathbf{x})-\mu(B_{\mathbf{x}}\sm E)\geq 0.9\cdot \eta\cdot \mmm(B_{\mathbf{x}})$$
we obtain a contradiction with the previous equation.
  \end{proof}

\section{Proof of Theorem \ref{mainth2}}
\label{proof2}

 In this section we construct a Cantor-like measure $\mu$ which satisfies the assumptions of Theorem \ref{mainth2}. 
 The main idea is that the construction steps leading to $\mu$ are similar to the ones of a standard Cantor set and measure, except that for some exceptional steps where we   impose that some annuli   carry the essential weight of the mass.

\subsection{Preliminaries}

Fix $0<\eta<1$,  $d-1<    \underline{d}  <  \overline{d}  < d$, and  put 
\begin{equation}\label{*ab*aa}
\ddd=\frac{\overline{d}-(d-1)}{\underline{d}-(d-1)} >1, 
\end{equation}
where the equality  is equivalent to
\begin{equation}\label{*ab*a}
(d-1)(1-\ddd)+\ddd \underline{d}=\overline{d}.
\end{equation}

\medskip

We call   
$\cad_{n}$, $n=0,1,...$  the family of half-open dyadic cubes of side length $2^{-{n}}$,
that is cubes $Q=\prod_{i=1}^{d}[k_{i}\cdot 2^{-n},(k_{i}+1)2^{-n})$, $k_{i}\in \Z$,
 $i=1,...,d$. Observe that $Q$ contains exactly one of its vertices, namely the one with coordinates $(k_{1}\cdot 2^{-n},...,k_{d}\cdot 2^{-n})$. We call this vertex the 
 {\it smallest vertex} of $Q$ and denote it by ${\mathbf v}_{\min,Q}$.
The sum of the coordinates of ${\mathbf v}_{\min,Q}$ is denoted by $s({\mathbf v}_{\min,Q})$, that is 
\begin{equation}
\label{sumvertex}
s({\mathbf v}_{\min,Q})=k_{1}\cdot 2^{-n}+...+k_{d}\cdot 2^{-n}.
\end{equation}

\medskip

\begin{definition}
\label{def-boundary}
For every cube $Q\in \mathcal{D}_m$, for every $n>m$
denote by $\partial_{n } Q\sse \cad_{n}$ the set of {\it $n $-boundary cubes} of $Q$, that is
those $Q'\in \cad_{n } $ which are included in $Q$, but at least one of their neighbors
is not included in $Q$. We denote by $d_{n }(Q) $ the number of $\cad_{n } $ cubes in
$\partial_{n } Q$.  

Out of the $2d$ faces of $Q$, we call $\ddno Q$ the face consisting   of those cubes $Q'\in \dd_{n} Q$ for which the first coordinate of its smallest vertex is $k_{1} \cdot 2^{-n}$: $\ddno Q$ will be called  the {\it smallest face} of $Q$.
 We also put $\partial_{m,1} Q=\{Q\}$.
\end{definition}

We denote by $d_{n,1 }(Q) $ the number of $\cad_{n } $ cubes in
$\partial_{n,1 } Q$.
It is clear that 
\begin{equation}\label{*dmnestb}
d_{n,1 }(Q) =2^{(d-1)(n-m)}.
\end{equation}



In the rest of this section, we construct a sequence of mass distributions $(\mu_{m})_{m \geq 1}$ which converges to a measure $\mu$ that will satisfy the assumptions of Theorem \ref{mainth2}.

We put $\mmm_{0}([0,1)^{d})=1,$ and for $Q\not =[0,1)^{d}$, $Q\in \cad_{0}$, 
we impose  that $\mmm_{0}(Q)=0.$

At the $m$th step, the mass distribution  $\mu_m$  will be defined  by fixing   the $\mu_m$-weight of every  cube $Q\in \cad_{m}  $, and this $\mu_m$-mass will be uniformly distributed inside every such $Q$. Then  $\mu_{m+1}$ will be a refinement of $\mu_m$ in the sense that 
\begin{equation}
\label{kolmo}
\mbox{for every $Q\in \cad_m$, $\mu_{m+1}(Q)=\mu_m(Q)$.}
\end{equation}
Due to Kolmogorov's extension theorem (see for example \cite{Oksendal}, \cite{Taomeas} or \cite{Lamperti})
  this ensures the weak convergence of $(\mu_m)$ to a measure $\mu$ defined on $\zu^d$.  

Set $\hhh^{*}=\sqrt \hhh>\hhh.$

Fix a constant $c_d>0$ so large  that 
\begin{align}
\label{cd1}
 c_{d} &\geq  2^{10d+1},\quad  \hhh^{*}>c_{d}^{-1}\\
 \label{cd2}
 1-\hhh>{ 1- \eta^{*}   } &\geq c_d^{-1} ,  \ \mbox{ and } \ \      \eta ^{*}2^{d+1+\ddd}          \leq c_d.
 \end{align}

Denote by   $\cad_{m}^+$  the subset of $\cad_m$ containing those cubes $Q\in \cad_m$   for which the measure $\mu_m(Q)> 0$.
The sequence of measures   $(\mu_m)_{m\geq 1}$  will satisfy that     for some $C>1$,   $\mbox{for every $m\geq 1$,}$ $\text{ for all }Q\in\cad_{m}^+ $,
\begin{equation*}
C^{-1} 2^{-m\overline{d}}\leq \mu_{m}(Q) \leq C  2^{-m \underline{d}}.
\end{equation*}

\medskip

We are going to    alternate between  two subdivision schemes.  The  subdivision scheme of type {\bf A} is meant to distribute quite uniformly the mass of a cube into some of its subcubes, while  the subdivision scheme of type {\bf B} will concentrate the mass of $Q$ into a very thin "layer" close to the boundary near the smallest face  of $Q$ and around its center.

\subsection{Subdivision scheme of type A}
\label{*subsectypA} 
 
 Assuming that $\mu_m$ is defined on $\mathcal{D}_m$, this scheme {\bf A} is applied to one individual cube  $Q \in \cad_m$ to define a measure  $\mu_{m+1}$  on the subcubes $Q'\in \mathcal{D}_{m+1}$ included in $Q$.  
 This subdivision scheme {\bf A}  distinguishes three cases:

\begin{enumerate}
\medskip
\item[(A1)]  If $\mu_{m}(Q)=0$,  then for any $Q'\sse Q$ with  $Q'\in \cad_{m+1}$, we put 
$\mu_{m+1}(Q')=0.$ 

\medskip
\item[(A2)]  If $2^{-d}\mu_{m}(Q)\geq  2^{-(m+1) \overline{d}}$ then for any 
  $Q'\sse Q$ with  $Q'\in \cad_{m+1}$, we set 
$\mu_{m+1}(Q')=2^{-d}\mu_{m}(Q).$

\medskip
\item[(A3)] 
If $2^{-d}\mu_{m}(Q)<  2^{-(m+1) \overline{d}}$, then we concentrate all the mass on the subcube $Q'\in \mathcal{D}_{m+1}$ included in $Q$ whose smallest vertex is the same as that of $Q$. In other words,  $\mu_{m+1}(Q')=\mu_{m}(Q)$ and ${\mathbf v}_{\min,Q}={\mathbf v}_{\min,Q'}$. 

\noindent
For  all the other  cubes
  $Q''\sse Q$, $Q''\in \cad_{m+1}$, we put 
$\mu_{m+1}(Q'')=0.$
\end{enumerate}

It is clear that  with this process $\mu_{m+1}(Q)=\mu_m(Q)$, so $\mu_{m+1}$ is indeed a refinement of $\mu_m$ on $Q$.

Remark that (A2) tends to spread the mass of $Q$ uniformly on its subcubes (hence to make the local dimension increase since $\overline{d}<d$) while (A3) tends to concentrate the mass  (hence to make the local dimension decrease from generation $m$ to generation $m+1$).

\begin{lemma}\label{*clstepaa} 
Assume that $\mu_m$ satisfies 
\begin{equation}\label{*m1*a}
 c_{d}^{-2} 2^{-m\overline{d}}\leq \mu_{m}(Q) \leq c_{d}^2 2^{-m \underline{d}},
\end{equation}
with $Q\in \cad _{m}$, and apply  subdivision scheme {\bf A} to define $\mu_{m+1}$ on the subcubes  $Q'\sse Q$, $Q'\in \cad_{m+1}$.

Then, for every  $Q'\in \cad_{m+1}$ with $Q'\sse Q$, such that  $\mu_{m+1} (Q')\neq 0$,  \eqref{*m1*a} holds with  the measure $\mu_{m+1}$ and generation   $m+1$, i.e.
\begin{equation}
\label{eqm+1}
 c_{d}^{-2}  2^{-(m+1)\overline{d}}\leq \mu_{m+1}(Q') \leq c_{d}^2 2^{-(m+1) \underline{d}}.
\end{equation}
\end{lemma}

\begin{proof}
Assume that we are in situation (A2). Hence,  initially we had 
$$ 2^d 2^{-(m+1) \overline{d}} \leq \mu_{m}(Q) \leq c_{d}^2  2^{-m \underline{d}}.$$
The construction ensures that for $Q'\sse Q$, $Q'\in \cad_{m+1}$,
$$     2^{-(m+1) \overline{d}} \leq \mu_{m+1}(Q' ) =2^{-d} \mu_m(Q) \leq 2^{-d} c_{d}^2 2^{-m \underline{d}} \leq c_{d}  ^22^{-(m+1)  \underline{d}}$$
which implies \eqref{eqm+1} (the last inequality holds since $\underline{d}<d$).

Assume that we are now  in situation (A3), which  implies that
$$ c_{d}^{-2}  2^{-m\overline{d}}\leq \mu_{m}(Q) <  2^d  2^{-(m+1) \overline{d}}<
2^d  2^{-(m+1) \underline{d}}.$$
Thus $\mu_{m+1}(Q')= \mu_{m}(Q)$ for one selected $ Q'\sse Q$, $Q'\in \cad_{m+1}$, and by \eqref{cd1}
$$  c_{d}^{-2} 2^{-(m+1)\overline{d}}\leq 
c_{d}^{-2} 2^{-m\overline{d}}\leq \mu_{m+1}(Q') <  2^d  2^{-(m+1) \underline{d}} <  c_{d}^22^{-(m+1) \underline{d}} .$$
\end{proof}

We prove now that if we apply scheme {\bf A} a sufficiently large number of times, then we obtain cubes $Q\in \cad_n$   which all satisfy
 \begin{equation}\label{*m2*a}
\mbox{ either } 2^{-n \overline{d}}\leq \mu_{n }(Q)< c_{d} 2^{-n \overline{d}},
 \text{ or }\mmm_{n}(Q)=0.
  \end{equation}

\begin{lemma}\label{*clstepaa2} 
Assume that $\mu_m$ satisfies  \eqref{*m1*a} with $Q\in \cad _{m}$, and apply     subdivision scheme {\bf A} to $Q$ to define $\mu_{m+1}$ on the subcubes  $Q'\sse Q$, $Q'\in \cad_{m+1}$,  then  apply     subdivision scheme {\bf A} to  all $Q'\subset Q$, $Q'\in \cad_{m+1}$,  to define  $\mu_{m+2}$ on all subcubes  $Q''\sse Q$, $Q''\in \cad_{m+2}$,
etc...

There exists an integer $\phi(Q) >m$ such that for every $n \geq \phi(Q) $  and  every cube $Q' \in \mathcal{D}_{n}$ with $Q' \subset Q$, \eqref{*m2*a} holds for $Q' $ and $\mu_{n}$.
\end{lemma}

\begin{proof}
We separate two cases depending on whether at step $m$
we need to apply (A2)   or (A3).

The second case can be reduced to the first. Indeed, taking into consideration \eqref{*m1*a}, suppose that we have
$$  c_{d}^{-2} 2^{-m\overline{d}}\leq \mu_{m}(Q) < 2^d 2^{-(m+1) \overline{d}}$$
and we start with subdivision (A3).
Then for $Q'\sse Q$, $Q'\in\mathcal{D}_{m+1}$, either $\mmm_{m+1}(Q')=0$, or
$$ c_{d}^{-2} 2^{-(m+1)\overline{d}}\leq \mu_{m+1}(Q')=\mu_{m}(Q)<
2^d 2^{-(m+1) \overline{d}} <  c_{d}^2 2^{-m \underline{d}}.$$
At the next step, we either have to apply division step (A3), or we can apply division step (A2). It is also clear that after finitely many steps we get to a situation when 
the first time  step (A2) must be applied. In that case, at level $n\geq m$, we have
exactly one $Q'\sse Q$, $Q'\in\mathcal{D}_{n}$ such that
\begin{equation}\label{*sA2}
2^d 2^{-(n+1)\overline{d}}\leq \mu_{n}(Q') <  c_d^2 2^{-n \underline{d}}
\end{equation}
and for all other descendants $Q''\sse Q$, $Q''\in\mathcal{D}_{n}$, $\mu_{n}(Q'')=0$. 
Then we can start an argument which is the same as if  we started with a subdivision (A2) from the very beginning. 

Observe that $\overline{d}<d$ and \eqref{*sA2} imply that
$2^{-n\overline{d}}<2^d2^{-(n+1)\overline{d}}\leq \mmm_n(Q')$.

For ease of notation we suppose that at step $m$ we can already start with a subdivision step (A2), that is \eqref{*sA2} holds with $m$ instead of $n$.

Now we apply (A2) to $Q$ and $\mu_m$, and iteratively to all subcubes of $Q$ of generation $n>m$, as long as $\mu_{n}(Q' ) \geq   2^d 2^{-(n+1) \overline{d}}$ for $Q' \in \cad_{n}$. 
 
 Observe that for a cube $Q'\in\mathcal{D}_{n}$ with $Q'\subset Q$, as long as 
$$\text{$2^{-d}\mu_{n}(Q')\geq  2^{-(n+1) \overline{d}}$, that is $ \mu_{n}(Q')\geq 2^d 2^{-(n+1) \overline{d}}
> 2^d 2^{-(n+1) {d}}=2^{-nd}$,}$$ 
the mass of every subcube  $Q''\subset Q'$ of next generation is such that $\mu_{n+1}(Q'') = 2^{-d} \mu_{n}(Q')$.
Hence  
\begin{equation}
\label{eqdec1}
\frac{\log \mu_{n+1}(Q'')}{ \log 2^{-(n+1)}}=\frac{-d\log 2 +\log \mu_{n}(Q')}{-(n+1) \log 2} > \frac{\log \mu_{n}(Q')}{ \log 2^{-n}}.
\end{equation}
This means that the local dimension increases from generation $n$ to generation $n+1$.

The construction ensures that  all the subcubes of $Q$ at a given generation $n>m$ have the same $\mu_{n}$-mass.  
 
 Further, by \eqref{eqdec1}, the sequence $ \frac{\log \mu_{n}(Q')}{ \log 2^{-n}}$ (for $Q'\subset Q$, $Q'\in \mathcal{D}_{n}$)  is strictly increasing. Assuming that this process (A2) is iterated a number of times very large when compared to $m$, we would have  $\frac{\log \mu_{n}(Q')}{ \log 2^{-n}}  \sim \frac{\log (2^{-dn})}{ \log 2^{-n}} =d$ so 
$\mu_{n}(Q' ) \sim 2^{-dn} <\!\!<  2^d 2^{-(n+1) \overline{d}}$, since $ \overline{d}<d$. 

Hence, after a finite number of iterations, we necessarily have $\mu_{n}(Q' ) <   c_{d}  2^{-n\overline{d}} $.

Call $ \phi(Q) \geq m$  the first integer such that  for all $Q'\subset Q$, $Q'\in \mathcal{D}_{\phi(Q) }$, 
$$  2^{-\phi(Q) \overline{d}} \leq  \mu_{\phi(Q) }(Q') \leq c_{d}   2^{-\phi(Q)  \overline{d}}.$$
\begin{remark}
In case we started with subdivision steps (A3) before getting to a subcube in which we could apply (A2), then we define $\phi(Q)$ starting from this subcube. 
Recall that for ease of notation at the beginning of this part of the argument,we supposed that we start with subdivision steps (A2) at the $m$th step.
\end{remark}
Recall that at generation $ \phi(Q) $, all the cubes $Q'\subset Q$, $Q'\in \mathcal{D}_{\phi(Q) }$, have the same $\mu_{\phi(Q) }$-mass.
This also shows that \eqref{*m2*a} holds at generation $\phi(Q).$

 Assume that \eqref{*m2*a} holds at generation $n\geq \phi(Q) $ for $Q'\in \cad_{n}$, $Q'\subset Q$. Then 
 \begin{itemize} 
\item
if $2^{-d}\mu_{n}(Q')\geq  2^{-(n+1) \overline{d}}$, then  we apply {(A2)} and for every $Q''\in \cad_{n+1}$, $Q''\subset Q'$,
$$    2^{-(n+1) \overline{d}}  \leq \mu_{n+1}(Q'')  \leq 2^{-d} c_{d}    2^{- n\overline{d}}  \leq   c_{d}   2^{-(n+1)\overline{d}}  . $$
\item
if $2^{-d}\mu_{n}(Q')\leq  2^{-(n+1) \overline{d}}$, then we apply {(A3)} and   for every $Q''\in \cad_{n+1}$, either $\mu_{n+1}(Q'')=0$ or 
$$  2^{-(n+1) \overline{d}}  \leq  2^{-n \overline{d}}  \leq \mu_{n+1}(Q'')  \leq  2^{d}    2^{-(n+1) \overline{d}} \leq  c_{d}  2^{-({n}+1)\overline{d}}  . $$
\end{itemize} 
In all cases, \eqref{*m2*a} holds at generation $n+1$. Hence the result.
\end{proof}

\subsection
{Subdivision scheme of type B}\label{*subsectypB}
\

Let $m \in \N$. Consider some $Q\in \cad_{m}$, and assume that,
\begin{equation}\label{*m2*a-bis}
 2^{-m \overline{d}}\leq \mu_{m }(Q)< c_{d}2^d 2^{-m \overline{d}}
  \end{equation}
   holds for  $\mu_m$ and $Q$ at generation $m$.

The purpose of the second subdivision scheme is to concentrate  the mass $\mu_{m}(Q)$  on two subparts of $Q$,  first in a thin region   close to the (inner part of the) boundary of $Q$ (very close to its smallest $\psi(m)$-face), and  second   around its center. 
More precisely, we will assign $\eta^{*}$   of the initial mass $\mu_{m}(Q)$ to part of an annulus very thin close to the border of $Q$ 
(on Figure \ref{mcafig1} this is the thin blue shaded rectangular region), and $  1-\eta^{*}$ in a small cube located around the center of $Q$
(on Figure \ref{mcafig1}, this is the small blue shaded central square). The remaining subcubes of $Q$ will receive zero $\mu$-mass.

\subsubsection[(B1)]
{Distributing (part of) the mass on the smallest face:}\label{*subsubB1}
\

Choose the smallest integer $\psi(m) $  such that 
\begin{equation}\label{*m2*b}
 2^{-(m+1) \ddd -1} \leq 2^{-\psi(m) }<  2^{-(m+1) \ddd }.
\end{equation}
Since $\delta >1$, we have $\psi(m)>m+1$.

Consider  $\partial_{\psi(m),1} Q$ the set of $ \psi(m)$-boundary cubes 
 on the smallest face of $Q$. Recalling \eqref{*dmnestb}, we have
\begin{equation}\label{*dmnest_old}
d_{\psi(m),1 } (Q)=2^{(d-1)(\psi(m)-m) }.
\end{equation}

%

For all $Q'\in \partial_{\psi(m),1} Q$, $Q'\in \cad_{\psi(m)}$, we put
\begin{equation}\label{*m3*a}
\mu_{\psi(m)}(Q')= \eta^* \cdot \frac{1}{d_{\psi(m),1}(Q)}\mu_{m}(Q).
\end{equation}

Combining \eqref{*ab*a}, \eqref{*m2*a-bis}, \eqref{*m2*b}, \eqref{*dmnest_old} and \eqref{*m3*a}, we see that for all $Q'\in \partial_{\psi(m),1} Q$,
\begin{align}
\nonumber
\mu_{\psi(m)}(Q') &    \geq   \eta^* 2^{-(d-1)(\psi(m)-m) }  2^{-m \overline{d}}\\ 
\nonumber
 &    =  \eta^* 2^{-(d-1)\psi(m) } 2^{-m( \overline{d} - (d-1))}\\ 
\nonumber
 &    = \eta^* 2^{-(d-1)\psi(m) } 2^{-m \delta ( \underline{d}- (d-1))}\\ 
\nonumber
 &    \geq \eta^*  2^{-(d-1)\psi(m) } 2^{-\psi(m)   ( \underline{d}- (d-1)) }\\ 
\label{ineg11}
 &>  c_d^{-1} 2^{-\psi(m)\underline{d}},
 \end{align} 
 where   \eqref{cd1} has been used for the last lower bound. 
 By \eqref{*m2*b}, $(m+1)\ddd+1\geq \psi(m)
 $, hence 
\begin{align}
\nonumber & -(d-1)\psi(m)-m \delta ( \underline{d}- (d-1))\\
\nonumber & \leq
 -\psi(m) \underline{d}+( \underline{d}- (d-1))\psi(m)-m \delta ( \underline{d}- (d-1))\\
 &\label{*psse}
 \leq 
 -\psi(m) \underline{d}+( \underline{d}- (d-1))((m+1)\ddd+1)-m \delta ( \underline{d}- (d-1))
\\
\nonumber &= 
  -\psi(m) \underline{d}+( \underline{d}- (d-1))(\ddd+1).
  \end{align}
Using  \eqref{*ab*aa}, \eqref{cd2}, \eqref{*m2*a-bis}, \eqref{*m2*b} and \eqref{*psse}, one   gets 
\begin{align}
\nonumber
\mu_{\psi(m) }(Q') 
\nonumber
& \leq \eta^*  2^{-(d-1)(\psi(m)-m) }  c_{d} 2^d 2^{-m \overline{d}}\\
\nonumber
& =\eta^*  2^d  c_d2^{-(d-1)\psi(m) } 2^{-m( \overline{d} - (d-1))}\\
\nonumber
&  \eta^* 2^d  c_d2^{-(d-1)\psi(m) } 2^{-m \delta ( \underline{d}- (d-1))}\\
\nonumber
& \leq\eta^* 2^d     c_{d}  2^{-\psi(m) \underline{d}}
\cdot 2^{(1+\ddd)({ \underline{d}- (d-1)})}
\\
\label{ineg11p}
&\leq  c_d^{2} 2^{-\psi(m)\underline{d}}.
 \end{align} 
 Finally,   for all $Q'\in \partial_{\psi(m),1} Q$ we have
\begin{equation}
\label{eqm'}
 c_d^{-2}2^{-\psi(m) \underline{d}} \leq \mu_{\psi(m) }(Q') \leq c_{d} ^22^{-\psi(m) \underline{d}}.
\end{equation}
In particular, these cubes satisfy \eqref{*m1*a}.

Intuitively, starting with a cube $Q$ such that $\mu(Q)\sim |Q|^{\overline{d}}$, we end up with many small cubes $Q'\subset Q$, all located on the border of $Q$, and such that $\mu(Q')\sim |Q'|^{\underline{d}}$. 

\medskip

\begin{figure}
\centering{
\includegraphics[width=1\textwidth]{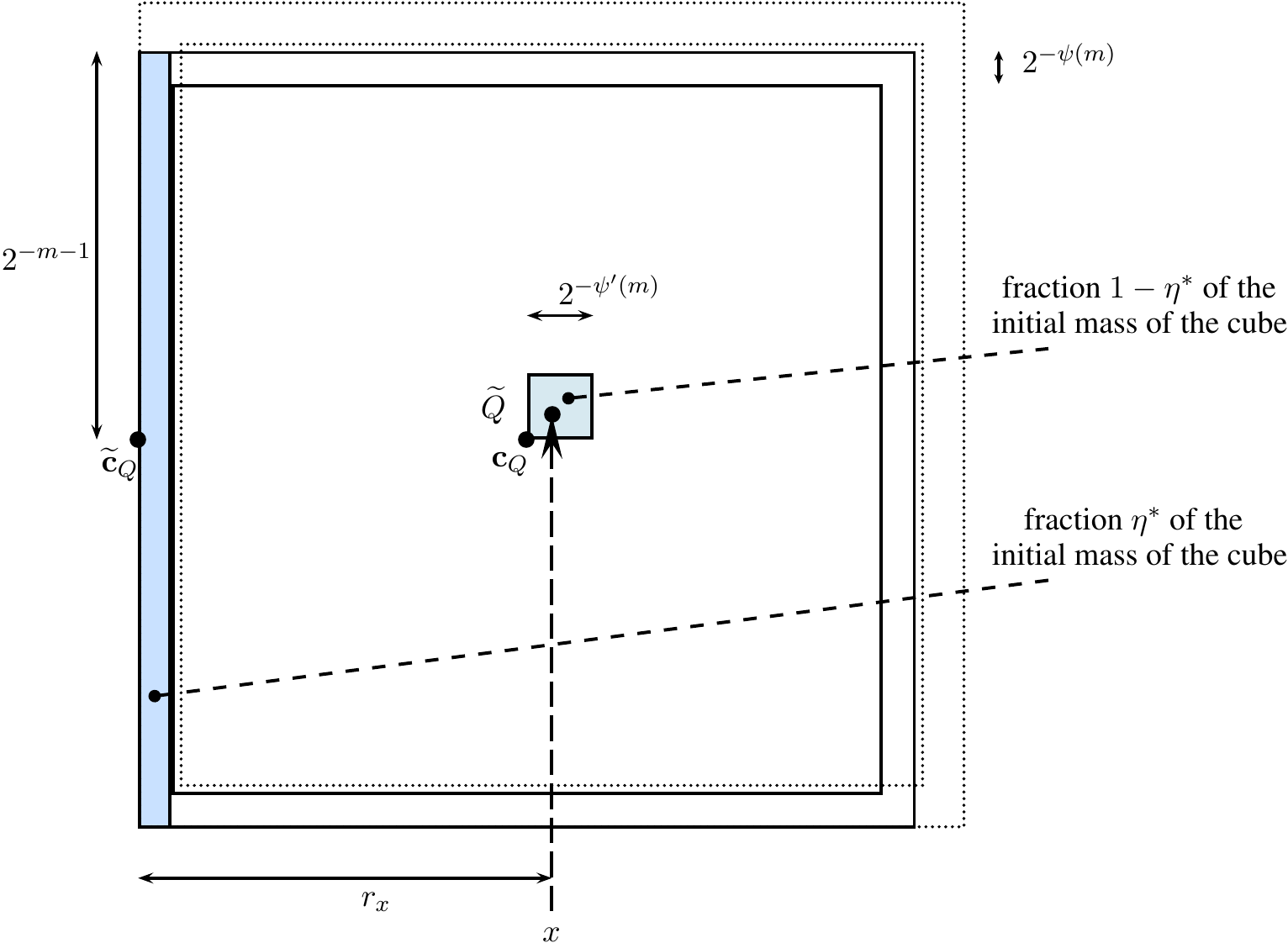}}
 \caption{Subdivision scheme $B$ for a cube $Q\in \cad_m$}
\label{mcafig1}
\end{figure}




\subsubsection[(B2)]
{Distributing part of the mass close to the center of $Q$:}
\label{*332}
\

Let $\psi'(m)  \geq m$  be the unique  integer satisfying
\begin{equation}\label{*m3*c}
    2^{-m \overline{d} -1 } \leq  2^{- \psi'(m) \underline{d} } \leq 2^{-m \overline{d}}. 
\end{equation}
Then it is easy to see that when $m$ becomes large,  $m <\psi'(m) < \psi(m)$  (intuitively,   $\psi'(m) \sim  m \overline{d} /\underline{d} $ while $\psi(m)\sim m\delta$ and by \eqref{*ab*a}, $\delta > \overline{d} /\underline{d}$).
We assume that $m$ is so large that 
\begin{align}
\label{eqpsi''}
& (1-2^{-(\ppp'(m)-m)}-2^{-(\ppp(m)-m)})^{d-1}
>(1-2^{-(\ppp'(m)-m)+1})^{d-1}>\hhh^{*}.
 \end{align}
 
We denote by ${\mathbf c}_{Q}$
the center of $Q$ and  by $ \widetilde{Q} $   the (unique) dyadic cube of generation $\psi'(m)  $ that contains 
${\mathbf c}_{Q} $.
Then, since we deal with dyadic cubes,  ${\mathbf c}_{Q} $ is the smallest vertex of $ \widetilde{Q}$, that is ${\mathbf c}_{Q}={\mathbf v}_{\min,\widetilde{Q}} $ .
 We put
\begin{equation}\label{*mass2}
\mmm_{\psi'(m)  }(\widetilde{Q} )=( 1-\eta^{*})
\mu_{m}(Q).
\end{equation}

 By using \eqref{cd1},  \eqref{cd2}, \eqref{*m2*a-bis}, \eqref{*m3*c}  and \eqref {eqpsi''},
we obtain
\begin{align}
\label{*mm3*c}
 \mmm_{\psi'(m)    }(\widetilde{Q} ) \geq  (1-\eta^*) 2^{-m\overline{d}}
\geq  ( 1-\eta^{*})  2^{-\psi'(m)   \underline{d}} \geq c_d^{-2}    2^{-\psi'(m)   \underline{d}}  \end{align}
and
$$  \mmm_{\psi'(m)   }(\widetilde{Q} )<  2^d c_{d} 2^{-m\overline{d}} < 2^{d+1} c_{d}  2^{-\psi'(m)   \underline{d}} \leq c_{d}^{2} 2^{-\psi'(m)   \underline{d}}. 
$$

We deduce that $\widetilde Q$ and $\psi'(m) $ satisfy equation \eqref{eqm'}, i.e.
\begin{align}
\label{*mm4*c} 
c_d^{-2}2^{-\psi'(m) \underline{d}} \leq \mu_{\psi'(m) }(\wq) \leq c_{d} ^22^{-\psi'(m) \underline{d}}.
\end{align}

 Hence one can apply subdivision scheme {\bf A} to it, iteratively, for all integers $n$
 such that $\psi ' (m)  < n \leq \psi(m) $. At the end of the process, by Lemma \ref{*clstepaa}, we get a 
 collection of cubes $Q'\in \cad_{\psi(m) }$ and a measure $\mu_{\psi(m) }$ such that 
 they all satisfy either $\mu_{\psi(m) }(Q')=0$, or 
\begin{eqnarray}
\label{eqm'2}
  c_{d}^{-2} 2^{-\psi(m) \overline{d}} \leq  \mmm_{\psi(m) }({Q'} ) \leq c_{d}^{2} 2^{- \psi(m) \underline{d}}.
\end{eqnarray}

\begin{definition}
\label{defcenter}
If $Q'\in \cad_{\psi(m) }$ is such that $Q'\subset \wq \subset Q$  (where $\wq$ is the cube of $\mathcal{D}_{\psi'(m)}$   containing   ${\mathbf c}_{Q}$),  
then $Q'$ is called {\bf a $B$-central cube at scale $\psi(m)$ associated to $Q\in \cad_m$}.
\end{definition}

By construction, 
\begin{equation}
\label{center}
\mu\Big( \bigcup_{ \substack{Q' \subset Q: \ \ Q' \text{ is a $B$-central}\\
\text{ cube at scale  }\psi(m)
} 
}Q'\Big) = ( 1-\eta^{*})\mu(Q).
\end{equation}

\begin{lemma}
\label{lemcenter}
If, for some large integer $m$,  $Q'\in \cad_{\psi(m) } $ is $B$-central, then for any $x\in Q'$, 
there exists $r_{x}$ such that $P_{\mu_{\psi(m)} }(x,r_{x},\delta,\eta)$ holds and
$$2^{-m-1}\leq r_{x}< 2^{-m-1}\cdot 1.125.$$
\end{lemma}

\begin{remark}
\label{rem2}
At inequalities \eqref{*mAxrxd} and \eqref{*mAxrxdd} in the next proof, we will use that  $\mu(B(x,2^{-m}))  = \mu(Q) $,
 which means essentially that $\mu_{\psi(m)}$ (and $\mu$) charges only the cube $Q$ and not its neighbors at generation $m$. This will be a consequence of our construction in Section \ref{secconstruction}.
\end{remark}

\begin{remark}
\label{rem2prime}
Lemma \ref{lemcenter} is stated for the measure $\mu_{\psi(m)}$, but it   also holds for the measure $\mu$ obtained at the end of the construction. This simply follows from \eqref{kolmo}.

\end{remark}

\begin{proof}

For simplicity, we write $\mu$ for $\mu_{\psi(m)}$. As explained above, this abuse of notation is justified by \eqref{kolmo}.

Let $Q'$ be a $B$-central cube, and $x\in Q'$. We shall prove that, for some $r_x>0$,   $\mu(A(x,r_{x}, \delta)) \geq \eta \cdot \mu(B(x,r_{x}))$.

Write $x=(x_{1},...,x_{d})$ and $Q=\prod_{i=1}^{d}[k_{i}\cdot 2^{-m},(k_{i}+1)\cdot 2^{-m})$.
See Figure \ref{mcafig1} where 
$x$ is marked with a dot in the $B$-central shaded blue cube, an arrow with dashed line
is pointing at $x$, the label $x$ is written at the bottom end of this arrow,
the distance $r_{x}$ is marked by a solid left-right arrow, and the boundary of 
$A(x,r_{x},\ddd)$ is shown with dotted lines.

Set $r_{x}:=x_{1}-k_{1}\cdot 2^{-m}$. Since  $\|x-{\mathbf c}_{Q}\|_{\oo}\leq 2^{-\psi'(m)}$ and ${\mathbf c}_{Q}=((k_{1}+\frac{1}{2})2^{-m},...,(k_{d}+\frac{1}{2})2^{-m})$,
we see that 
$$2^{-m-1}\leq r_{x}\leq 2^{-m-1} + 2^{ -\psi'(m)}  .$$

By construction, $B(x,r_x)$ contains  the largest part  of  $\dd_{\ppp(m),1}Q$. Indeed, if $Q'$ is a cube of $\dd_{\ppp(m),1}Q$ containing one  $y\in\dd_{\ppp(m),1}Q$ with $\|x-y\|_{\oo}> r_x$, then   $ \| {\mathbf c}_{Q} -y\|_{\oo}  \geq  \|x-y\|_{\oo}  -  \| { \mathbf c}_{Q} -x  \|_{\oo}  \geq  r_x -   \| { \mathbf c}_{Q} -x  \|_{\oo} \geq 2^{-m-1} - 2^{-\psi'(m)} $.   
Denote by $\widetilde{\mathbf c}_{Q}$ the projection of ${\mathbf c}_{Q}$
onto the smallest face of $Q$, see Figure  \ref{mcafig1}.
Since  the supremum norm is used,
the cubes   of  $\dd_{\ppp(m),1}Q$ that may not intersect  $B(x,r_x)$ are thus located outside of $B(\widetilde{\mathbf c}_{Q}, 2^{-m-1} - 2^{-\psi'(m)} )$. Observe that the
 intersection of $B(\widetilde{\mathbf c}_{Q}, 2^{-m-1} - 2^{-\psi'(m)} )$ with the smallest face of $Q$ is a $(d-1)$-dimensional cube of side length $2^{-m} - 2^{1-\psi'(m)}$, and that 
 its projection onto the smallest face of $Q$ is of $(d-1)$-dimensional volume
 $ (2^{-m} - 2^{1-\psi'(m)} )^{d-1}$.

 Observe also that $ r_x^\delta \geq  2^{-\delta(m+1)} > 2^{-\psi(m)}$ so all the above cubes $Q'$  belonging simultaneously  to $B(x,r_x)$ and $\dd_{\ppp(m),1}Q$  also are included in  $A(x,r_{x},\ddd)$. 
 These cubes are forming a single layer, and their projection onto
 the smallest face of $Q$ is of $(d-1)$-dimensional volume ${2^{-(d-1)\psi(m)}} $.
 
 Recalling that   
$d_{\ppp(m),1}(Q) =2^{(d-1)(\psi(m)-m)}$, 
 we obtain that $A(x,r_{x},\ddd)$ contains more than 
 \begin{align*}
  \frac{   (2^{-m} - 2^{1-\psi'(m)} )^{d-1} } {2^{-(d-1)\psi(m)}}  
 & \geq d_{\ppp(m),1}(Q)   (1-2^{-(\ppp'(m)-m)+1})^{d-1}\\
  & \geq \eta^{*}d_{\ppp(m),1}(Q)
  \end{align*}
many cubes from $\dd_{\ppp(m),1}Q$, where \eqref{eqpsi''} is also used.
Hence, recalling \eqref{*m3*a}, 
\begin{equation}\label{*mAxrxd}
\mmm(A(x,r_{x},\ddd))>\eta^{*}d_{\ppp(m),1}(Q)\cdot \eta^{*}\frac{1}{d_{\ppp(m),1}(Q)}\mmm_{m}(Q)
\end{equation}
$$=(\eta^{*})^{2}\mmm_{m}(Q)=\hhh\mmm(Q).$$
 Finally, by Remark \ref{rem2}, our construction   ensures that 
 $\mu(B(x,r_{x}))  \leq \mu(Q) $. Thus
 \begin{equation}\label{*mAxrxdd}
\mu(A(x,r_{x},\delta)) \geq \eta   \mu(B(x,r_{x})),
\end{equation}
i.e. $P_{\mu }(x,r_{x} ,\delta,\eta)$ holds.

Finally, the fact that $2^{-m-1}\leq r_{x}< 2^{-m-1}\cdot 1.125$ follows from $2^{-m-1}\leq r_{x}\leq 2^{-m-1} + 2^{ -\psi'(m)}  $ and $2^{-\psi'(m)+m}$ tends to zero when $m\to +\infty$.
\end{proof}

\subsubsection[(B3)]
{Giving a zero-mass to the other cubes, and defining the measure $\mu_{\psi(m)}$ on $Q$}\label{part333}
\

For the remaining cubes at generation $\psi(m)$, i.e. those $Q'\sse Q$, $Q'\in \cad_{\psi(m)}$, $Q'\not\in \partial_{\psi(m),1}(Q)$ and $Q'\not \sse \widetilde{Q} $,  we put  $\mu_{\psi(m)}(Q')=0.$

\medskip

The measure  $\mmm_{\psi(m) }(Q')$ is now defined for all
$Q'\sse \widetilde{Q} $, $Q'\in \cad_{\psi(m)}$.

 \medskip

Observe that by \eqref{*m3*a}, the properties of the subdivision scheme A and \eqref{*mass2},
we have
$$\sum_{Q'\sse Q,\ Q'\in\cad_{\psi(m)}}\mu_{\psi(m)}(Q')=\mu_{m}(Q),$$
i.e.  we indeed distributed the mass of $Q$ onto the cubes $Q'\sse Q$, $Q'\in\cad_{\psi(m)} $.

\medskip

\subsubsection[(B4)]
{Defining  inside $Q$ the measure $\mu_n$ for $m<n<\psi(m)$}\label{part334}
\

Forgetting for a while the details of the construction, and just focusing on the result, starting from $Q\in \cad_m$  with a given $\mu_m$-weight satisfying \eqref{*m2*a}, we end up with a measure $\mu_{\psi(m)}$ well defined on all cubes $Q'\subset Q$, $Q'\in \cad_{\psi(m)}$. 
 
Since we jumped from level $m $ to $\psi(m)$,  we can easily define
$\mmm_{n}$ for $m<n<\psi(m)$, {but only inside $Q$}. Indeed, for such an integer $n$, the measure $\mu_n$ is simply defined by using $\mu_{\psi(m)}$ as ``reference" measure:   for $Q'\in \cad_{n}$, $Q'\subset Q$, we set  
\begin{equation}\label{*mmmq}
\mmm_{n}(Q')=\sum_{Q''\sse Q',\  Q''\in \cad_{\psi(m)} }\mu_{\psi(m) }(Q'').
\end{equation}
It is easily checked that this definition is consistent with the definition of $\mu_{\psi(m)} $ that we gave in  \eqref{*mass2} for instance.


%
%
%

Next lemma shows that all the "intermediary" measures $\mu_n$ so defined share the same scaling properties as $\mu_m$ and $\mu_{\psi(m)}$.

\begin{lemma}\label{*clstepab2} 
Assume that $\mu_{m}$ satisfies  \eqref{*m2*a-bis} for some $Q\in \cad _{m}$, and apply the subdivision scheme {\bf B} to define $\mu_{m+1},...,\mu_{\psi(m)}$ on the subcubes  of $  Q$ of generation $m+1,...,\psi(m)$.

Then:
\begin{enumerate}
\item
 for every $n\in \{m,..., \psi(m)\}$, for every $Q'\in \cad_{n} $ such that $Q'\sse Q$ and $\mu_n(Q')\neq 0$,  \eqref{*m1*a} holds  for $Q'$ with  the measure $\mu_{n}$.
\item
for every cube $Q'\in \cad_{\psi(m)} $ such that  $Q'\subset Q$ and $\mu_{\psi(m)}(Q')\neq 0$, there exists $n\in \{m,...,\psi(m)\}$ and a (unique) cube $Q_n\in \cad_n$ such that $Q'\subset Q_n\subset Q$ and  \eqref{eqm'} holds for  $\mu_n $ and $Q_n$.
\end{enumerate}
\end{lemma}

\begin{proof} 

(i) The cases where $n=m$ and $n=\psi(m)$ follow from  \eqref{*m2*a-bis}, \eqref{ineg11}, \eqref{ineg11p}, and  \eqref{eqm'2}.

Suppose $m<n<\ppp(m)$. Two cases are separated depending on whether we deal with the border or the $B$-central cubes.

Consider first  $Q''\in \dd_{n,1}Q$, i.e. a cube located on the border of $Q$. Choose an arbitrary $Q'\in \dd_{\ppp(m),1}Q$,
$Q'\sse Q$. By \eqref{eqm'} and $d-1<\underline{d}$, one  obtains
$$\mmm_{n}(Q'')=2^{(\ppp(m)-n)(d-1)}\mmm_{n}(Q')\leq 2^{(\ppp(m)-n)(d-1)}
c_{d}^{2}2^{-\ppp(m)\underline{d}}$$
$$<c_{d}^{2}2^{(\ppp(m)-n)\underline{d}}\cdot 2^{-\ppp(m)\underline{d}}<
c_{d}^{2}2^{-n\underline{d}}.$$
For the estimate from below \eqref{cd1},  \eqref{*m2*a-bis}, \eqref{*dmnest_old}, \eqref{*m3*a} and \eqref{*mmmq} yield
$$\mmm_{n}(Q'')=2^{(\ppp(m)-n)(d-1)}\hhh^{*}2^{(m-\ppp(m))(d-1)}\mmm_{m}(Q)\geq
\hhh^{*}2^{(m-n)(d-1)}\cdot 2^{-m\overline{d}}$$
$$>\hhh^{*}2^{(m-n)\overline{d}}\cdot 2^{-m\overline{d}}=\hhh^{*}
2^{-n\overline{d}}>c_{d}^{-2}2^{-n\overline{d}}, $$
hence  \eqref{*m1*a} holds for $Q''$.

For the $B$-central cubes, suppose that $m<n<\ppp'(m)$ and  $Q''\in \cad_{n}$, $Q''\supset \widetilde{Q}$, where $\widetilde{Q}$ was defined after \eqref{eqpsi''}.
 Then $$\mmm_{n}(Q'')=\mmm_{\ppp(m)}(\widetilde{Q})=\mmm_{\ppp'(m)}(\widetilde{Q})=(1-\hhh^{*})\mmm_{m}(Q)$$
 and by \eqref{*mm4*c}
 \begin{equation}\label{*mQta}
 \mmm_{n}(Q'')=\mmm_{\ppp'(m)}(\widetilde{Q})\leq c_{d}^{2}2^{-\ppp'(m)\underline{d}}<c_{d}^{2}2^{-n\underline{d}}.
 \end{equation}
 On the other hand by \eqref{cd2} and \eqref{*m2*a-bis}
\begin{equation}\label{*mQtb}
\mmm_{n}(Q'')=(1-\hhh^{*})\mmm_{m}(Q)>(1-\hhh^{*})2^{-m\overline{d}}
>(1-\hhh^{*})2^{-n\overline{d}}>c_{d}^{-2}2^{-n\overline{d}},
\end{equation}
hence  \eqref{*m1*a} holds for $Q''$.

 Finally, when $ \ppp'(m) \leq n <\psi(m)$ and $Q''\in \mathcal{D}_n$ and $Q''\subset \widetilde {Q}$, the fact that  \eqref{*m1*a} holds for $Q''$ follows from the application of subdivision scheme {\bf A} to the cube $\widetilde{Q}$ which satisfies  \eqref{*m1*a}.

 \medskip

 Part  (ii) follows from \eqref{eqm'} and \eqref{*mm4*c}.
\end{proof}

Observe that the construction ensures that, as claimed at the beginning of this section, 
$$
\sum_{Q'\in\partial_{\psi(m),1} Q}\mu_{\psi(m)}(Q')= \eta^{*} \cdot   \mu_{m}(Q) .
$$

%
%

\subsection{Construction of the measure of Theorem \ref{mainth2}.} \label{secconstruction}\

Recall that  $\mmm_{0}$ is the Lebesgue measure on the cube $\zu^d$. By definition $\mu_0$ satisfies \eqref{*m1*a}.

\medskip

{\bf Step 1:}

We apply Subdivision Scheme A, $m'_1$ times to the cube $\zu^d$, where $m'_1 \geq \phi([0,1]^{d})$ is such that
 Lemma \ref{*clstepaa2}  and \eqref{eqpsi''} simultaneously hold for $m=m'_1$. We obtain a measure $\mu_{m'_1}$ defined on cubes of $\cad_{m'_1}$, such that for every  $Q\in  \cad_{m'_1}$, either $\mu_{m'_1}(Q)=0$ or  
the first half of \eqref{*m2*a} holds true with  $\mu_{m'_1}$. 
 
Now, for each cube $Q'\in \cad_{m'_{1}}$, we select the cube $Q$ of generation $m_1=m'_1+1$ located at its smallest vertex and we set $\mu_{m_1}(Q) = \mu_{m'_1}(Q')$, so that $\mu_{m_1}$ satisfies \eqref{*m2*a-bis}.
 
 This last step ensures that the cubes at generation $m_1$ supporting $\mu_{m_1}$ are isolated, and Lemma \ref{lemcenter} (together with Remark \ref{rem2}) applies. 
 
%
%
%
%
%
%
%
%
%
%

\medskip

Now we are able to iterate the construction:
 
\medskip

{\bf Step $2k$:}

We apply Subdivision scheme B to all cubes $Q\in \cad_{m_{2k-1}}$.  Call $m_{2k}=\psi(m_{2k-1})$.

 We obtain a measure $\mu_{m_{2k}}$ defined on $\cad_{m_{2k}}$ such that the properties of Lemma \ref{*clstepaa2}  hold for all dyadic cubes $Q\in \cad_n$, $m_{2k-1}\leq n\leq m_{2k}$.

\medskip

{\bf Step $2k+1$:}

We apply Subdivision scheme A  to all cubes of generation $m_{2k}$. By Lemma \ref{*clstepaa2}, for each cube $Q\in \cad_{m_{2k}}$, there exists an integer $\phi(Q)$ such that  for $n\geq \phi(Q)$, for every cube $Q' \in \mathcal{D}_{n}$  \eqref{*m2*a} holds for $Q'$ and $\mu_{n}$. 

 Setting $m'_{2k+1}=\max \{\phi(Q): Q\in \cad_{m_{2k}} \}$, we are left with a measure $\mu_{m'_{2k+1}}$ such that for all cubes $Q'\in \cad_{m'_{2k+1}}$ \eqref{*m2*a} holds for $Q'$ and $\mu_{m'_{2k+1}}$.

Setting $m_{2k+1}=m'_{2k+1}+1$, for each  cube $Q'\in \cad_{m'_{2k+1}}$, we select the cube $Q$ of generation $m_{2k+1}$ located at its smallest vertex and we set $\mu_{m_{2k+1}}(Q) = \mu_{m'_{2k+1}}(Q')$, so that $\mu_{m_{2k+1}}$ satisfies \eqref{*m2*a-bis} and the cubes  at generation $m_{2k+1}$ supporting $\mu_{2k+1}$ are isolated.

 
\medskip

We are now ready to construct the set and the measure satisfying the conditions of Theorem \ref{mainth2}.

Call $Q_n(x)$  the unique dyadic cube $Q\in \cad_n$ that contains $x$.

\begin{proposition}\label{*clsta}
The sequence $(\mu_n)_{n\geq 1}$ converges to a measure $\mu$ which is supported by a Cantor-like set $\mathcal{C} $ defined by
$$\mathcal{C} = \bigcap_{n\geq 1} \  \bigcup_{Q\in \cad_{n}: \mu_n(Q) \neq 0}  \ Q.$$

For every $x\in \mathcal{C}$, for every $n$,
\begin{equation}
\label{encadrement1}
 c_{d}^{-2} 2^{-n\overline{d}}\leq \mu_{n}(Q_n(x)) \leq c_{d}^2 2^{-n \underline{d}},
\end{equation}
and there exist two strictly increasing sequences of integers $(\overline{j}_n(x))_{n\geq 1}$  and $( \underline{j}_n(x))_{n\geq 1}$ satisfying
\begin{equation}
\label{encadrement2}
 2^{-\overline{j}_n(x) \overline{d}}\leq \mu_{\overline{j}_n(x) }(Q)< c_{d} 2^{-\overline{j}_n(x) \overline{d}}
\end{equation}
and
\begin{equation}
\label{encadrement3}
 c_{d}^{-2}  2^{-\underline{j}_n(x) \underline{d}}\leq \mu_{\underline{j}_n(x) }(Q)< c_{d}^2 2^{-\underline{j}_n(x) \underline{d}}
\end{equation}

\end{proposition}

\begin{proof}
Inequalities \eqref{encadrement1}, \eqref{encadrement2} and \eqref{encadrement3}  follow immediately from \eqref{eqm+1} of Lemma \ref{*clstepaa},
 Lemma \ref{*clstepaa2} and Lemma \ref{*clstepab2}.
\end{proof}

Recall the that the upper and lower local dimensions of the measure $\mmm$ are defined by
$$\overline{d}_\mu(x)=\limsup_{r\searrow 0}\frac{\log (\mmm(B(x,r)))}{\log r} \text{ and }\underline{d}_\mu(x)=\liminf_{r\searrow 0}\frac{\log (\mmm(B(x,r)))}{\log r}. $$

From the previous proposition, we easily deduce the following property.

\begin{corollary}
For every $x\in \mathcal{C}$, $\underline{d}_\mu(x) =\underline{d}$  and $\overline{d}_\mu(x) =\overline{d}$. In particular, 
the measure $\mu$ satisfies 
$$ \underline{d} = \underline{\dim}(\mu) <  \overline{\dim}(\mu) = \overline{d}.$$
\end{corollary}

\begin{proposition}\label{*clcentr}
For $\mmm$-almost every $x$, there exist infinitely many  integers $n$
such that $Q_{2n}(x)$ is a  $B$-central cube at Step $m_{2n}$.
\end{proposition}

The construction ensures that $\mu$-almost all points are regularly located in a $B$-central cube, in the sense of Definition \ref{defcenter}.

\begin{proof}
For $n\geq 1$, call
$$\mathcal{A}_n=\left\{x\in \mathcal{C}: 
\begin{array}{l} Q_{m_{2n}}(x)  \mbox{ is a $B$-central cube at generation } m_{2n}\\  \mbox{associated with a cube }  Q\in \cad_{m_{2n-1}} \end{array} \right\}$$.

By construction, and recalling \eqref{center}, we get $\mu(\mathcal{A}_n)= 1- \eta^{*}$.

Also, it is clear from the uniformity of the construction that the sequence $(\mathcal{A}_n)_{n\geq 1}$  is independent when seen as events with respect to the probability measure $\mu$:  for every finite set of integers $(n_1,n_2,...,n_p)$, 
$$\mu(\mathcal{A}_{n_1}\cap \mathcal{A}_{n_2}\cap \cdots \cap \mathcal{A}_{n_p})=  (1-\eta^{*})^p.$$

Applying the (second) Borel--Cantelli lemma (see for example \cite{Grimmett01}, Section 7.3),
 we obtain that $\mu$-almost every point belongs to an infinite number of sets $\mathcal{A}_n$. Hence the result.
\end{proof}

\begin{remark}\label{*rem344}
Observe that the same proof gives that $\mu$-almost every point belongs to an infinite number of sets  $\mathcal{A}_n^c$.
\end{remark}

\begin{corollary}
The measure $\mu$ satisfies $\mu(E_{\mu}(\delta,\eta))  = 1$.
\end{corollary}

This   follows from Proposition \ref{*clcentr} and Lemma \ref{lemcenter}.

%
%

%

\section{Proof of Theorem \ref{mainth3}}\label{proof3}

We deal with the case where $\underline{\dim} \,\mu \leq d-1$.
In this situation, as stated in  Theorem \ref{mainth3}, there is no more   restriction
on $\ddd$. However we can argue almost like
 in the proof of Theorem \ref{mainth2}. 
 
Fix an arbitrary $\ddd>1$.
 
 As before, we are going to use various subdivision schemes to build a measure fulfilling our properties.
 
 The {Subdivision scheme of type {\bf A}} in Section \ref{*subsectypA}   is left unchanged.
 
 The {Subdivision scheme of type {\bf B}} in Section \ref{*subsectypB}  requires some adjustments (in particular, one cannot use   \eqref{*ab*aa} any more). The problem comes from the fact  that when $\underline{\dim}\,\mu$ is less than $d-1$, when trying to spread the mass of a given cube $Q \in \mathcal{D}_m$ such that $\mmm_{m}(Q)\sim 2^{-m\overline{d} }$  to the  smaller cubes $Q'$  located on its smallest face $\partial_{\psi(m),1}Q$, it is not possible to impose that  $\mmm_{m}(Q')\sim 2^{-m\underline{d} }$  for all $Q' \in \partial_{\psi(m),1}Q$, since 
$$ d_{\psi(m) ,1}(Q) 2^{-\psi(m)\underline{d}} \sim 2^{-(d-1)(\psi(m)-m)} 2^{-\psi(m)\underline{d}} 
>\!\!\!>  2^{-m\overline{d}} .$$
%
 In other words, the mass of the initial cube $Q$ is not large enough to give the sufficient weight to each $Q'$. So we introduce a Subdivision scheme of type {\bf C} to solve this issue.
%
%

\medskip

We discuss   the modifications needed to adapt Subsection \ref{*subsectypB} to this situation $\underline{d}< d-1$, and give the main ideas to proceed - some proofs are omitted, since they are exactly similar to those of Section \ref{proof2}.

\subsection{Subdivision scheme of type {\bf C}}\label{*subsectypC}

As explained above, a new subdivision scheme is introduced, by essentially modifying a little bit Subdivision scheme of type {\bf B}.

Assume that  $Q\in \cad_{m}$ satisfies \eqref{*m2*a-bis}. The integer  $\psi(m)$ is defined by \eqref{*m2*b}, as in the previous Section. We know that  \eqref{*dmnest_old} holds.
As   argued above, one can see that using \eqref{*m2*a-bis}
\begin{equation}\label{*th3*3a}
\frac{1}{d_{\psi(m),1}(Q)}\mmm_{m}(Q)<
 c_{d}2^{d}2^{-m\overline{d}}\cdot 2^{-(d-1)(\psi(m)-m)}
\end{equation}
$$\leq c_{d}2^{d}2^{-m\underline{d}}\cdot 2^{-\underline{d}(\psi(m)-m)}= c_{d}2^{d}2^{-\psi(m)\underline{d}}.$$

First, the way the mass is distributed on the border of $Q$ (Subsection 
\ref{*332}) is modified as follows.

\subsubsection{Distributing (part of) the mass on the smallest face:}\label{*subsubC1}
 Two cases are separated. 

\mk

{\bf $\bullet$  Case {\bf  $\underline{d} < d-1 \leq \overline{d}$}:}    We set
$$d_{\psi(m),1,+}(Q)=d_{\psi(m),1}(Q)=2^{(d-1)(\ppp(m)-m)}$$
and 
for any $Q'\in \partial_{\psi(m),1,+}Q:=\partial_{\psi(m),1}Q$, put
\begin{equation}\label{*m3*ab}
\mmm_{\psi(m)}(Q')=\eta^{*} \frac{1}{d_{\psi(m),1,+}(Q)}\mmm_{m}(Q).
\end{equation}
In this case we also define $\partial_{\psi(m),1,0}Q=\ess$, and $d_{\psi(m),1,0}(Q)=0.$

For all $m<n\leq \ppp(m)$ and $Q'\in \dd_{n,1,+}Q:=\dd_{n,1}Q$,   using
$\underline{d}\leq d-1\leq \overline{d}$ and \eqref{*m2*a-bis}, one gets
\begin{align}
\label{*I*a}
\mmm_{n}(Q')& =\eta^{*}\frac{1}{d_{\psi(m),1,+}(Q)}\cdot \mmm_{m}(Q)\cdot 2^{(d-1)(\ppp(m)-n)}\\
\nonumber
&<\eta^{*} 2^{-(d-1)(\ppp(m)-m)} \cdot c_{d}2^{d}2^{-m\overline{d}}\cdot 2^{(d-1)(\ppp(m)-n)}\\
\nonumber
&=\eta^{*} c_{d}2^{d}2^{m(d-1)}\cdot 2^{-m\overline{d}}\cdot 2^{-(d-1)n} \leq \eta^{*} c_{d}2^{d}\cdot 2^{-n\underline{d}} \\
\nonumber
&< c_{d}^{2}2^{-n\underline{d}}.
\end{align}
Similarly,
\begin{align}\label{*II*a}
\mmm_{n}(Q') & >\eta^{*} 2^{-(d-1)(\ppp(m)-m)} \cdot 2^{-m\overline{d}}\cdot 2^{(d-1)(\ppp(m)-n)}\\
\nonumber
&=\eta^{*} 2^{-(d-1)(n-m)}2^{-m\overline{d}}\geq 
\eta^{*} 2^{-\overline{d}(n-m)}2^{-m\overline{d}}  \geq\eta^{*} 2^{-n\overline{d}}\\
\nonumber
&
> c_{d}^{-2}2^{-n\overline{d}}.
\end{align}
This means that \eqref{*m1*a} will remain true for these cubes. 

\noindent To resume, the scheme is in this case the same as before.

\mk

{\bf $\bullet$   Case {\bf  $\underline{d}\leq \overline{d} < d-1$}:} An extra care is needed. 

 \sk

Recall Definition \ref{def-boundary}. Consider first the cubes $Q' \in \dd_{m+1,1}Q $.

\begin{enumerate}
\medskip
  \item  If $\eta^* 2^{-(d-1)}  \mmm_{m}(Q)\geq 2^{-(m+1)\overline{d}}$,
then for any  $Q'\in \dd_{m+1,1}Q$, define $\mmm_{m+1}(Q')=\eta^*2^{-(d-1)}\mmm_{m}(Q').$
For $Q'\sse Q$, $Q'\in \cad_{m+1}$, but $Q'\not \in \dd_{m+1,1}Q$, set $\mmm_{m+1}(Q')=0.$ 
 \item  If $\eta^* 2^{-(d-1)}\cdot \mmm_{m}(Q)< 2^{-(m+1)\underline{d}}$,    consider the only cube $\widehat{Q}$ with ``maximal" smallest vertex $\mathbf{v}_{\min,\widehat{Q}}$ among
those cubes satisfying  $\widetilde{Q}\in Q$, $ \widetilde{Q}\in \dd_{m+1,1}Q$. ``Maximal" means with largest possible coordinates - this makes sense since the borders of the cube are parallel to the axes.
If it is easier to understand this way,  for such a vertex, 
 the sum of its coordinates 
$s(\mathbf{v}_{\min,\widehat{Q}})$ (defined in \eqref{sumvertex}) is maximal.
Then put    $\mmm_{m+1}(\widehat{Q})=\eta^*\mmm_{m}(Q)$.

 For the other cubes $Q'\sse Q$, $Q'\in \cad_{m+1}$,  set $\mmm_{m+1}(Q')=0.$ 
  \end{enumerate}

Next we iterate the process.
 Suppose that $m+1<n\leq \ppp(m)$, and that $\mmm_{n-1}(Q')$ is defined for all $Q'\in
\dd_{n-1,1}Q.$ 

\begin{enumerate}
\medskip
\item[(${{C}}$1)] If $\mmm_{n-1}(Q')=0$  then for all subcubes $Q''\in \cad_{n}$ of $Q'$, put  $\mmm_{n}(Q'')=0.$
 \item[(${{C}}$2)] If $2^{-(d-1)}\mmm_{n-1}(Q')\geq 2^{-n\overline{d}}$
then for any $Q''\sse Q'$ with $Q''\in \dd_{n,1}Q$, we set
$\mmm_{n}(Q'')=2^{-(d-1)}\mmm_{n-1}(Q').$
For $Q''\sse Q',$ $Q''\in \cad_{n}$ but $Q''\not \in \dd_{n,1}Q$,
we set $\mmm_{n}(Q'')=0.$ 
 \item[(${{C}}$3)]  If $2^{-(d-1)}\cdot \mmm_{n-1}(Q')< 2^{-n\overline{d}}$,  then, as above, we select the cube 
  \begin{equation}\label{*III*a}
 \widehat{Q}\in Q', \  \widehat{Q}\in \dd_{n-1}Q
 \end{equation}
  with maximal $s(\mathbf{v}_{\min,\widehat{Q}})$ among those  cubes satisfying \eqref{*III*a}. Next, we set   $\mmm_{n}(\widehat{Q})=\mmm_{n-1}(Q')$.
  
 For the other cubes $Q''\sse Q',$ $Q''\in \cad_{n}$, we impose  $\mmm_{n}(Q'')=0.$ 
 \end{enumerate}

The motivation for the choice of $\widehat{Q}$ is that    $\widehat{Q}$ is located (if we look at our two-dimensional Figure \ref{mcafig1}) in the upper left "Northwest" direction from
$\mathbf{c}_{Q}$. We select $\widehat{Q}$ in the ``upper" corner of $Q'$ on the boundary of $Q$ in order to get 
as many as possible of the charged cubes into $A(x,r_x,\delta)  $ in Lemma \ref{lemcenter}. 
Recall that on Figure \ref{mcafig1} the central cube $\widetilde{Q}$ 
with smallest vertex  $\mathbf{c}_{Q}$ is also located "above" (in the direction "Northeast") from this vertex and $x$ is located in $\widetilde{Q}$.

 We repeat the above steps for $n=m+1,...,\ppp(m)$ and denote by $\dd_{n,1,+}Q$ those cubes in $\dd_{n,1}Q$ for which $\mmm_{n}(Q')>0.$
The adjustment at step $n=m+1$ implies that 
we distribute a mass of $\eta^{*}\cdot \mmm_{m}(Q)$ on the $\dd_{\ppp(m),1,+}Q$
cubes and
in this case \eqref{*m3*ab} holds as well.

\medskip

By \eqref{*m2*a-bis} initially verified by $Q$ and the first step of our induction, for $n=m$ and  $n=m+1$,   one has 
\begin{equation}\label{*IV*a}
\eta^{*}\cdot 2^{-n\overline{d}}<\mmm_{n}(Q)<c_{d}2^{d}\cdot 2^{-n\overline{d}}.
\end{equation}
Suppose that $Q'\in \dd_{n,1,+}Q$   satisfies \eqref{*IV*a}.
Consider $Q''\in \dd_{n+1,1,+}Q$ such that  $Q''\sse Q'$.

If  Step (${{C}}$2) was used to define $\mmm_{n+1}(Q'')$, then on the one hand, 
$\mmm_{n}(Q')\geq 2^{(d-1)}2^{-(n+1)\overline{d}}$, and
hence 
\begin{equation}\label{*IV*c}
\mmm_{n+1}(Q'')\geq 2^{-(n+1)\overline{d}}.
\end{equation}
On the other hand, since   $\overline{d}\leq d-1$,
one sees that 
\begin{equation}\label{*IV*b}
\mmm_{n+1}(Q'')=2^{-(d-1)}\mmm_{n}(Q')<c_{d}2^{d}\cdot 2^{-(d-1)}2^{-n \overline{d}}\leq
c_{d}2^{d}2^{-(n+1)\overline{d}}.
\end{equation}

\medskip

If  (${{C}}$3)  was used, then for the only $\widehat{Q}\in \dd_{n+1,1,+}Q$ satisfying
$\widehat{Q}\sse Q'$, one has 
\begin{equation}\label{*V*a}
\mmm_{n+1}(\widehat{Q})=\mmm_{n}(Q')<2^{(d-1)}\cdot 2^{-(n+1)\overline{d}}<c_{d}2^{d}\cdot 2^{-(n+1)\overline{d}}.
\end{equation}
On the other hand, from \eqref{*IV*a} it also follows that
\begin{equation}\label{*V*b}
\eta^{*} 2^{-(n+1)\overline{d}}<\eta^{*} 2^{-n\overline{d}}<\mmm_{n}(Q')=\mmm_{n+1}(\widehat{Q}).
\end{equation}
Thus,  by induction,
\eqref{*IV*a} holds true. Hence \eqref{*m1*a} holds for any $Q'\in \dd_{n,1,+}Q$ for any $n=m+1,...,\ppp(m).$

\subsubsection[(C2)]
{Distributing part of the mass close to the center of $Q$:}
\label{*332bis}
\

  Step \ref{*332}  of Scheme {\bf B} is also modified for Scheme {\bf C}.

If $\overline{d} =\underline{d} $, then put $\ppp'(m)=m+10$ and
   assume     that \eqref{cd1}  holds.

If $\overline{d} >\underline{d} $, then we can define $\ppp'(m)$ as in \eqref{*m3*c}.

It is not necessarily true any more that $m\leq \ppp'(m) \leq\ppp(m)$. Indeed, recall that intuitively,    $\psi'(m) \sim  m \overline{d} /\underline{d} $ and $\psi(m)\sim m\delta$, but now  $\delta >1$ can be such that $1<\delta < \overline{d} /\underline{d}$.

Hence let us introduce  $\PPP(m)=\max\{ \ppp'(m),\ppp(m) \}.$

As before, call  $\wq$   the cube of $\mathcal{D}_{\psi'(m)}$   containing the center    ${\mathbf c}_{Q}$ of $Q$.

\begin{definition}
\label{defcenterbis}
If $Q'\in \cad_{\Psi(m) }$ is such that $Q'\subset \wq \subset Q$  (where $\wq$ is the cube of $\mathcal{D}_{\psi'(m)}$   containing   ${\mathbf c}_{Q}$),  $Q'$ is called {\bf a $C$-central cube at scale $\Psi(m)$ associated to $Q\in \cad_m$}.
\end{definition}

If $\PPP(m)=\psi(m) \geq \ppp'(m)$, then we   proceed analogously to Section \ref{*332}, i.e.  we put as in \eqref{*mass2} $
\mmm_{\psi'(m)  }(\widetilde{Q} )=( 1-\eta^{*})
\mu_{m}(Q)$, and apply subdivision scheme {\bf A} to $\widetilde{Q}$ and its subcubes until generation $\psi(m)$.

\smallskip

If $\PPP(m)=\ppp'(m)$ and $\psi'(m)>\psi(m)$ then we also set   as in \eqref{*mass2} $
\mmm_{\psi'(m)  }(\widetilde{Q} )=( 1-\eta^{*})
\mu_{m}(Q)$, i.e. we concentrate all the mass that was not spread on $\partial _{m,1}Q$ onto $\widetilde{Q}$.
 But in this situation in Subsection \ref{*subsubC1}
 the measures $\mu_n$ are completely defined inside $Q$ only for $m\leq n\leq \psi(m)$.
Hence,   we apply Subdivision scheme {\bf A} to the cubes $Q'\in \partial_{\psi(m),1,+}Q$ to distribute $\mmm_{\ppp(m)}(Q')$ onto some subcubes of $Q'$
and to define $\mmm_{n}$ on $Q'$ for $\ppp(m)<n\leq \PPP(m)=\psi'(m).$

By an immediate application of Lemma \ref{*clstepaa}, since $\mmm_{\ppp(m)}(Q')$
satisfied \eqref{*m1*a},  the same inequality \eqref{*m1*a} (with $m$ replaced by $n$) remains true for all  $n\in (\ppp(m),\PPP(m)]$
for any $Q''\sse Q'$, $Q''\in \cad_{n}^{+}$.

 In particular, the following analog of Lemma \ref{lemcenter} holds.
 \begin{lemma}
\label{lemcenterbis}
If, for some large integer $m$,  $Q'\in \cad_{\Psi(m) } $ is $C$-central, then for any $x\in Q'$, 
there exists $r_{x}$ such that $P_{\mu_{\Psi(m)} }(x,r_{x},\delta,\eta)$ holds and
$$2^{-m-1}\leq r_{x}< 2^{-m-1}\cdot 1.125.$$
\end{lemma}
\begin{proof}
We discuss only the changes    in the argument  of the original proof of Lemma \ref{lemcenter} when $\psi(m)<\Psi(m)=\psi'(m)$.

In this new  situation, $A(x,r_x,\delta) $ contains more than $$(1-2^{-(\ppp'(m)-m)+1})^{d-1}d_{\Psi(m),1}(Q)>\eta^{*} d_{\Psi(m),1}(Q)$$
many cubes from $\partial_{\Psi(m),1}Q$. Now even if we had to use Step (${{C}}$3), 
our choice of
the cube with maximum $s(\mathbf{v}_{\min,\widehat{Q}})$ after 
\eqref{*III*a}
implies that
$$(1-2^{-(\ppp'(m)-m)+1})^{d-1}d_{\Psi(m),1,+}(Q)>\eta^{*} d_{\Psi(m),1,+}(Q)$$
holds in this case as well. 
This finally yields
$$\mmm(A(x,r_x,\delta) )>\eta^{*}\cdot \eta^{*} \mu_m(Q)>\hhh\cdot \mu_m(Q).$$
\end{proof}

  \subsubsection{Giving a zero-mass to the other cubes, and defining the measure $\mu_{\Psi(m)}$ on $Q$}\label{part333bis}
 
 The measure $\mu_{\Psi(m)}$ is extended inside $Q$ as $\mu_{\Psi(m)}$  in   subsection \ref{part333}.

\subsubsection{Defining  inside $Q$ the measure $\mu_n$ for $m<n<\Psi(m)$}\label{part334bis}
\
 
 The measures $\mu_n$ are also defined as in Subsection \ref{part334}.  
 
 \sk
 
\begin{lemma}\label{*clstepab2bis} 
Assume that $\mu_{m}$ satisfies  \eqref{*m2*a-bis} for some $Q\in \cad _{m}$, and apply the subdivision scheme {\bf C} to define $\mu_{m+1},...,\mu_{\Psi(m)}$ on the subcubes  of $  Q$ of generation $m+1,...,\Psi(m)$.

Then:
\begin{enumerate}
\item
 for every $n\in \{m,..., \Psi(m)\}$, for every $Q'\in \cad_{n} $ such that $Q'\sse Q$ and $\mu_n(Q')\neq 0$,  \eqref{*m1*a} holds  for $Q'$ with  the measure $\mu_{n}$.
\item
for every cube $Q'\in \cad_{\Psi(m)} $ such that  $Q'\subset Q$ and $\mu_{\Psi(m)}(Q')\neq 0$, there exists $n\in \{m,...,\Psi(m)\}$ and a (unique) cube $Q_n\in \cad_n$ such that $Q'\subset Q_n\subset Q$ and  \eqref{eqm'} holds for  $\mu_n $ and $Q_n$.
\end{enumerate}
\end{lemma}

\begin{proof} 
The proof is   similar to that of Lemma \ref{*clstepab2}, up to some minor modifications that are left to the reader.
%
%
%
%
%
%
%
%
\end{proof}

\subsection{Construction of the measure of Theorem \ref{mainth2}.} \label{secconstructionbis}\

The measure $\mu$ is built exactly as in Section \ref{secconstruction}.

Proposition \ref{*clcentr} 
can be proved in this case as well.

\begin{proposition}\label{*clcentrbis}
For $\mmm$-almost every $x$, there exist infinitely many  integers $n$
such that $Q_{2n}(x)$ is a  $C$-central cube at Step $m_{2n}$.
\end{proposition}

The conclusions and the arguments are similar to those developed in Section \ref{secconstruction}, we only sketch the proof to get Theorem \ref{mainth2}

\medskip

As a consequence of Proposition \ref{*clcentrbis},   $\mmm$-a.e. $x$ belongs to a $C$-central cube infinitely often, and  \eqref{*mm4*c} holds infinitely often.
For such an $x$,  there exists an increasing sequence of integers  $(\underline{j}_n(x))_{n\geq 1}$  such that the following version 
of \eqref{encadrement3} holds
\begin{equation}\label{encadrement3b}
c_{d}^{-2} 2^{-\underline{j}_n(x)\underline{d}}\leq \mmm_{\underline{j}_n(x)}(Q)<c_{d}^{2}2^{-\underline{j}_n(x)\underline{d}}.
\end{equation}
By Lemma \ref{*clstepaa} and the construction, there exists another increasing  sequence of integers  $(\overline{j}_n(x))_{n\geq 1}$
satisfying \eqref{encadrement2}.

Then part (i) of Lemma \ref{*clstepab2bis} yields   the dimension estimate \eqref{encadrement1}, which concludes the proof.


 \section{Lemma \ref{lemmcor} and the proof of Theorem \ref{mainth4}}\label{proof4}

\subsection{Intersection of thin Euclidean annuli}

The idea of the proof of Theorem \ref{mainth4} is based on the observation that if
balls (in the $2$-dimensional  case, disks) of comparable radii $\sim 2^{-n}$ are centered not too close  (in the next lemma, the distance
between their centres  is at least $2^{-5n}$), then the annuli corresponding to these balls
are intersecting each other in a set of small diameter, see Figure \ref{mcafig2}.  This follows from  the   strictly convex  shape of the corresponding balls.
It is illustrated by Lemma \ref{lemmcor} below, which prevents   that measures with different upper and lower dimensions charge   thin annuli.

%

\begin{lemma}\label{lemmcor}
There exists an integer  $N_{corr}$   such that  if $n\geq N_{corr}$ then for every  $2^{-n-1}\leq r_{1},r_{2}\leq 2^{-n}$ and 
\begin{equation}\label{*eqlemcorb}
2^{-5n}\leq ||{z_{1}-z_{2}}||_{2}\leq 2^{-n}/30 ,
\end{equation}
 $ {A}(z_1,r_{1},30 )  \cap {A}(z_2,r_{2},30)$ consists of at most two connex  sets, each of them is of diameter, $D^{*}_{corr,n}$ less than $24\cdot 2^{-13.5n}$.
\end{lemma}

\begin{figure}
\centering{
\includegraphics[width=1\textwidth]{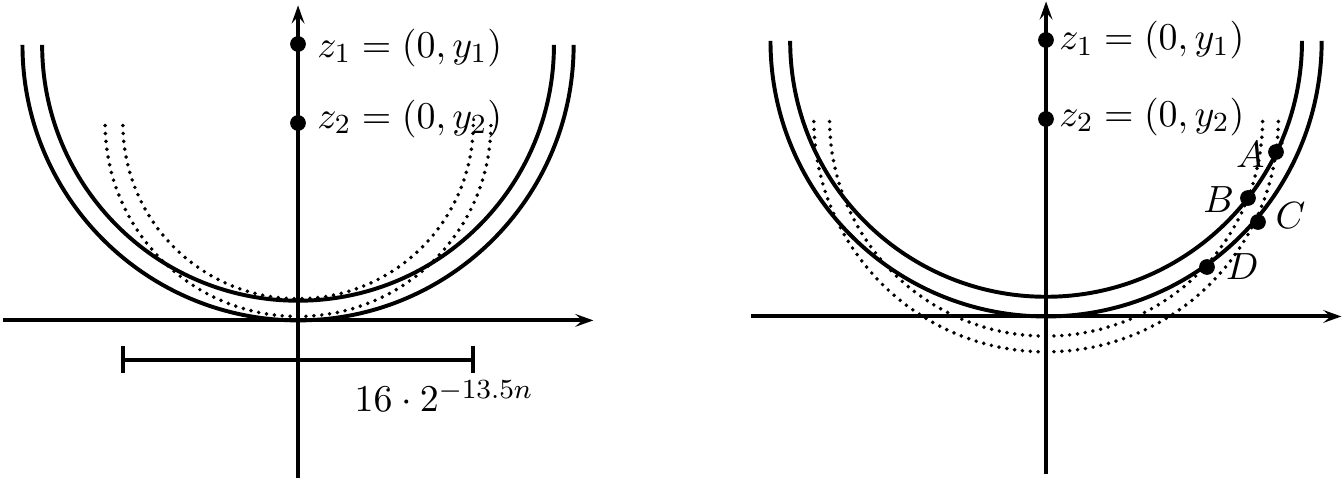}}
\caption{Positions of intersecting annuli.}
\label{mcafig2}
\end{figure}


This type of estimations appear at several places. For example from Lemma 3.1 of Wolff's survey \cite{WolffKakeyasur} one gets an estimate of $D^{*}_{corr,n}$ of the form $C_{W}2^{-12.5n}$ with a constant $C_{W}$
not depending on $n$. The order of this estimate is slightly smaller than ours.

We point out that the order $2^{-13.5n}$ is optimal. Indeed, taking
$$\car= {A}((0,2^{-n}),2^{-n},30 )  \cap {A}((0,2^{-n}-2^{-5n}),2^{-n}-2^{-5n},30),$$
 then one can verify that the diameter of $\car$ can be estimated from below by $C\cdot 2^{-13.5n}$ with a constant not depending on
$n$. For this, considering a triangle with sides $a=2^{-n}-2^{-5n}$, $b=2^{-5n}$ and
$c=2^{-n}-2^{-30n}$, it is easily seen that half of the diameter of $\car$ is larger than the altitude  $m_{b}$ of the triangle perpendicular to the side $b$.
Using Heron's formula the area of the triangle is $A=\sqrt {s(s-a)(s-b)(s-c)}$ with
$s=(a+b+c)/2$ and $m_{b}=2A/b$. Plugging in the above constants, one obtains the announced estimate. 

Since the notation of Lemma 3.1 of  \cite{WolffKakeyasur} is different from ours, we detail a bit the way one can obtain an estimate for $D^{*}_{corr,n}$ by using that lemma.
Let $d_{W}=|z_{1}-z_{2}|+|r_{1}-r_{2}|$ and $\DDD_{W}=||z_{1}-z_{2}|-|r_{1}-r_{2}||$.
We can suppose that $r_{1}\geq r_2$, and the assumptions  of Lemma \ref{lemmcor} 
imply that $r_{2}\geq r_{1}/2$. The argument of Lemma 3.1 of  \cite{WolffKakeyasur}
gives an estimate
$$D^{*}_{corr,n}\leq C_{W}'\frac{r_{1}^{-30n}}{\sqrt{(r_{1}^{-30n}+\DDD_{W})(r_{1}^{-30n}+d_{W})}}\text{with $C'_{W}$ independent of $n$.}$$
The assumptions of Lemma \ref{lemmcor} imply that $0\leq \DDD_{W}$ and $2^{-5n}\leq d_{W}$, and these inequalities cannot be improved.
Hence the above estimate implies
$$D^{*}_{corr,n}\leq C^{*}_{W}r_1^{-30n}r_{1}^{15n}r_{1}^{2.5 n}=C_{W}r_{1}^{-12.5n}\leq C_{W}2^{-12.5n}.$$

A similar order estimate can be obtained from Marstrand's paper \cite{Marcirc}.
In \cite{Boumax} a similar type of question is studied but it is less straightforward 
which order one can obtain for $D^{*}_{corr,n}$.

\medskip

We now   turn to the main result and prove Theorem \ref{mainth4}. Finally,  Lemma \ref{lemmcor} is proved in Section \ref{sseclemcor}.

\subsection{Proof of Theorem \ref{mainth4}}
 
Without limiting generality we can suppose that the Borel probability  measure $\mmm$
is supported on $[0,1]^{2}$.

Proceeding towards a contradiction, suppose that $\mmm(E_\mu(30,\eta))>0$.

For ease of notation, let $ {E}=E_\mu(30,\eta)$.
Since  $ {E}$ is fixed  in the rest of the proof, $\mmm( {E})$ will be regarded as a positive constant.

Since  ${\overline{\mathrm{dim}}\, } \mmm=\overline{d} \in [0.89,2]$ and  ${\underline{\mathrm{dim}}\, } \mmm=\underline{d}\in [0.89, {\overline{d}}] $, for $\mmm$-a.e. $x\in  {E}$ there exist
$\rho_x >0 $  such that 
 \begin{equation}
 \label{*th4*3*a}
\mbox{for any $0<r\leq \rho_x$, } \ \ \  r^{\overline{d}+0.01 }\leq  \mmm(B(x,r))\leq r^{{\underline{d }} -0.01} \leq r^{0.88}.
\end{equation}
For those $x$s for which $\rho_x$, as defined above, does not exist, we set $\rho_x =0$.

We can also suppose that for each $x$ we select $\rho_x$ as half of the 
supremum of those $\rho$s for which \eqref{*th4*3*a} holds with
$\rho$ instead of $\rho_x$.
This way it is not too difficult to see that the mapping $\rho : x\in \zu^2 \mapsto \rho_x \in  \R^+$ is  
Borel $\mmm$-measurable.

\medskip

The first step consists of finding a ball containing points of $E$ with a very precisely controlled behavior, see Lemma \ref{lem-ball}
 below.
To prove it, let us start with    a simple technical lemma.

\begin{lemma}\label{*lemdp}
Suppose $\widetilde{E}\subset \R^2 $ is a $\mmm$-measurable set, and let  $\rho :\R^2\to \R$,  $x\mapsto \rho_x$ be a $\mmm$-measurable function such that $\mu(\{x \in \widetilde{E}: \rho_x>0\}) =\mu(\widetilde{E})$.  

Then for any $0<\ggg<1$, for $\mmm$-a.e. $x\in \widetilde{E}$, 
 there exists $R_{x}>0 $  such that for any $0<r<R_{x}$,
 \begin{equation}\label{*ldp*a}
 \mmm\big(\{ x'\in B(x,r)\cap \widetilde{E}: \rho_{x'} \geq R_{x} \}\big)>\ggg\cdot \mmm(B(x,r)).
 \end{equation}
\end{lemma}

\begin{proof}
Consider $\widetilde{E}_{n}=\{ x\in\widetilde{E}:\rho_x>\frac{1}{n} \}$, for $n\geq 1$. Then $\mmm(\widetilde{E}\sm \cup_{n \in \N}\widetilde{E}_{n})=0$ by assumption. 

Fix now $n\in\N$. By Corollary \ref{*Ma214}
$$\frac{\mmm(\widetilde{E}_{n}\cap B(x,r))}{\mmm(B(x,r))}\to 1 \text{ for $\mmm$-a.e. }x\in\widetilde{E}_{n}.
$$
For those $x\in\widetilde{E}_{n}$ for which the above limit holds true,  it is thus enough to  choose $0<R_{x}<\frac{1}{n}$
 such that 
 $$\frac{\mmm(\widetilde{E}_{n}\cap B(x,r))}{\mmm(B(x,r))}>\ggg \text{ holds for any } 0<r<R_{x}.$$ \end{proof}

Using Lemma \ref{*lemdp} with $\ggg=1-\frac{\eta}{2}$ applied to  $ {E}$, for $\mmm$-a.e. $x\in  {E}$ there exists an  $R_{x}>0$ such that for any $0<r<R_{x}$,
\begin{equation}\label{*th4*6*b}
\mmm(\{ x'\in B(x,r)\cap E:\rho_{x'} \geq R_{x} \})> \Big (1-\frac{\hhh}{2}\Big )\mmm(B(x,r)).
\end{equation} 
%
%

Consider a large natural number $N_{0}>10$, whose precise value will be chosen later.
 
 By using the definition of $ {E}=E_\mu(30,\eta)$, for $\mmm$-a.e. $x\in  {E}$ one can select 
 $0<r_{x}\leq \min \{ 2^{-N_{0}},R_{x}/10, \rho_x/10 \}$ such that
\begin{equation}\label{*th4*3*b}
\mmm(A(x,r_x,30) )\geq \hhh\cdot \mmm(B(x,r_{x})).
\end{equation}

Recalling that all $r_x$s are less than $2^{-N_0}$, and that $\mu$-a.e. $r_x$ is strictly positive, there exists at least one integer   $n_0 \geq N_{0}$  such that 
\begin{equation}\label{*th4*4*a}
\mmm(\{ x\in {E}: 2^{-n_0 -1}\leq r_{x}<2^{-n_0 } \})\geq \frac{1}{10n_0 ^{2}}\mmm({E}).
\end{equation}

Consider now the covering of $\{ x\in {E}: 2^{-n_0 -1}\leq r_{x}<2^{-n_0 } \}$ by the balls $\{B(x,r_x ):  x\in {E},\,  2^{-n_0 -1}\leq r_{x}<2^{-n_0 } \}$. 
 By the measure theoretical version of   Besicovitch's   covering theorem (see \cite{Bene}, Theorem 20.1 for instance), there exists a constant $C_2>0$, depending only on the dimension 2, such that one can extract a finite family of disjoint balls of radius $4\cdot 2^{-n_0 }$, denoted by 
$\cab=\{ B_{i}:i=1,..,M  \}$, such that, calling $\widetilde{E} = {E}\cap \cup_{B_{i}\in \cab}B_{i}$,  one has $\mu(B_i)>0$  for every $i$, and 
   \begin{equation}\label{*th4*4*b} 
\mmm(\widetilde{E} _{n_0 })\geq \frac{ {C}_2 }{n_0 ^{2}}
\mmm( {E}), \mbox{ where } \widetilde{E} _{n_0 }=\{ x\in\widetilde{E} :2^{-n_0 -1}\leq r_{x}<2^{-n_0 } \}.
\end{equation}

\begin{lemma}
\label{lem-ball}
There exists a constant $C>0$ depending only on the dimension and a ball $B^*=B(x^*, 2^{-n_0}/100)$  such that 
 \begin{equation}\label{*th4*7*b}
 \mmm\Big(B \big(x^*  ,2^{-n_0 }/100 \big)\cap \widetilde{E} _{n_0}  \Big) \geq
 \frac{C}{n_0^{2}} 2^{-2{n_0} } \mmm( {E}).
 \end{equation}
\end{lemma}

\begin{proof}

Since $\widetilde{E}\subset \zu^2$,  the balls of $\cab$ 
are disjoint and are of radius less than $2^{-10}$,  their cardinality $M$  is less than $2^{2n_0}$, and there exists one ball $B_i$ such that 
\begin{equation}
\label{eqlemma1}
\mu(B_i \cap \widetilde{E} _{n_0 })\geq  \frac{1}{2^{2n_0+2}} \mu( \widetilde{E} _{n_0 }) \geq \frac{ {C}_2 }{2^{2n_0} n_0 ^{2}} \mmm({E}).
\end{equation}

Write $B_i=B(x_i, r_{x_i})$.
 Since  $ r_{x_i}/ 100< 2^{-n_0 }/100 < 2^{-n_{0}-1}\leq r_{x_i}$, there exists a constant $ {C}_3>0 $ which depends only on the dimension such that for some $x^*\in B_i \cap \widetilde{E}$, the ball with center $x^*$ and radius $2^{-n_0 }/100$ supports a proportion $C_3>0$ of the $\mu$-mass of the initial ball, i.e.
  \begin{equation}\label{*th4*5*a}
\mu(B(x^*, 2^{-n_0 }/100))  \geq \mmm(B(x^* ,2^{-n_0}/100) \cap \widetilde{E}_{n_0} )\geq  {C}_3 \mmm(B_i\cap\widetilde{E}_{n_0}).   
 \end{equation}

The last statement simply follows from the fact that $B_i$ can be covered by finitely many balls of radius $2^{-n_0 }/100$, this finite number of balls being  bounded above independently of $x_i$ and $r_{x_i}$.

%
 This and \eqref{eqlemma1} imply the result.
\end{proof}

 Further,  as a second step, we  seek for  a lower estimate of the   number $M^*_{n_0}$ of disjoint balls $B(y_{j},2^{-5n_0})$, $j=1,...,M^{*}_{n_0}$
 such that $y_{j}\in B(x^*, 2^{-{n_0}}/100)\cap \widetilde{E} _{n_0}.$  
The rest of the proof consists of
 showing that   annuli centered at $y_{j}$, $j=1,...,M^*_{n_0}$, will charge $B(x^*, 4\cdot 2^{-{n_0}})$ with too much measure, yielding a contradiction.  Lemma \ref{lemmcor} will play a key role here.

\medskip

 \begin{lemma}
 \label{lem_minor}
 When $N_0$ is sufficiently large, for every $n_0\geq N_0$, call    $M_{n_0} ^*$ the maximal number of disjoint balls $B(y_{j},2^{-5n_0})$, $j=1,...,M^{*}_{n_0}$
 such that $y_{j}\in B(x^*, 2^{-{n_0}}/100)\cap \widetilde{E} _{n_0}$. Then $M_{n_0}^* \geq 2^{2.1 n_0}$.
 \end{lemma}
 \begin{proof}
 
 First, observe that, when $y_{j}\in B(x^*, 2^{-{n_0}}/100)\cap \widetilde{E} _{n_0},$      by \eqref{*th4*3*a}
 \begin{equation}\label{*th4*8*a}
 \mmm(B(y_{j},2^{-5{n_0}}))
 \leq 2^{-0.88\cdot 5{n_0}}= 2^{ -4.4{n_0}}.
  \end{equation}

Next, there exist two  positive constants $ {C}_{3}$
and $ {C}_{4}$ depending only on the dimension  such that $B(x^*, 2^{-{n_0}}/100) \cap \widetilde{E} _{n_0}$ is covered by $C_3$ families $\mathcal{F}_1$, ..., $\mathcal{F}_{C_3}$  containing pairwise disjoint balls of the form $B(y_{j},2^{-5n_0})$. 
At least one of these families, say  $\mathcal{F}_1$, satisfies that 
$$\sum _{B(y_{j},2^{-5n_0}) \in \mathcal{F}_1} \mu(B(y_{j},2^{-5n_0})) \geq \frac{1}{C_3}  \mu( B(x^*, 2^{-{n_0}}/100)\cap \widetilde{E} _{n_0}) .$$

By \eqref{*th4*7*b}, $\sum _{B(y_{j},2^{-5n_0}) \in \mathcal{F}_1} \mu(B(y_{j},2^{-5n_0}))  \geq  \frac{C}{C_3 n_0^{2}} 2^{-2{n_0} } \mmm( {E})$. Then, from \eqref{*th4*8*a} one deduces that 
\begin{equation}\label{*th4*8*b}
M_{n_0}^{*}\geq  \frac{C}{C_3 n_0^2} 2^{-2{n_0} }   \mmm( {E})  \frac{1}{ 2^{-4.4{n_0}} }\geq   2^{2.1{n_0}},
\end{equation} 
when $N_{0}$ is  chosen sufficiently large.
 \end{proof}

%
%

Next, as a third step, we study the annuli  $A(y_i,r_{y_i},30) $ associated with the points $y_i$, $i=1,..., M^*_{n_0}$.  Observe that these points $y_i$  satisfy the assumptions of Lemma \ref{lemmcor} and  in particular equation \eqref{*eqlemcorb} with $n=n_0$, 
as soon as $N_{0}\geq N_{corr}$.

 \begin{lemma}
 \label{lemma55}
For every $x\in B(x^*, 2^{-n_0}/100)\cap \widetilde E_{n_0}$, set 
\begin{equation}\label{*Atdef}
\widetilde{A}(x,r_x,30)= \Big\{ x'\in A(x,r_x, 30)\cap {E}:\rho_{x'} >10r_x \Big\}.
\end{equation}
Then,   for some constant $C>0$ that depends only on the dimension one has 
\begin{equation}\label{*th4*10*b}
\mmm \big(\widetilde{A}(x,r_x,30)\big ) \geq \frac{\hhh}{2-\hhh} \cdot \frac{C}{n_0 ^{2}} \mmm( {E})\cdot 2^{-2{n_0} }.
\end{equation}
 \end{lemma}
 \begin{proof}

Since $r_{x}\leq R_{x}/10$, from \eqref{*th4*6*b} and \eqref{*th4*3*b} we infer that
for $\mmm$-a.e. $x\in {E}$,  one has
\begin{align*}
\mmm \big(\widetilde{A}(x,r_x,30)\big ) = \mmm(\{ x'\in A(x,r_x,30)\cap E: \rho_{x'} >10r_x \})  \geq  \frac{\hhh}{2}\mmm(B(x,r_x)).
\end{align*}
The last inequality holds for every $r_x$ such that  \eqref{*th4*6*b} holds true.

%
%
%

Recalling that for $x\in \widetilde{E}_{n_0} $,   $2^{-{n_0} -1}\leq r_{x}<2^{-{n_0} }$, one deduces that 
for any $x\in B(x^*, 2^{-{n_0} }/100)\cap \widetilde{E} _{n_0} $, 
 $ B(x^*, 2^{-{n_0} }/100)\sse B(x,r_{x}/4)$.

 Hence, by \eqref{*th4*7*b},
\begin{equation}\label{*th4*10*a}
\mmm(B(x,r_{x}/4))\geq \mmm( B(x^*, 2^{-{n_0} }/100)\cap
\widetilde{E} _{n_0} )\geq \frac{C}{n_0 ^{2}}\mmm( {E})\cdot 2^{-2{n_0} }.
\end{equation}

It is also clear that $B(x,r_{x}/4)\cap \widetilde{A}(x,r_x,30) =\ess$
for such $x$s. So, the fact that   
$$\mmm(\widetilde{A}(x,r_x,30) )\geq \frac{\hhh}{2}\mmm(B(x,r_{x}))
\geq \frac{\hhh}{2}\Big ( \mmm(\widetilde{A}(x,r_x,30) )+\mmm(B(x,r_{x}/4)) \Big )$$
implies by using \eqref{*th4*10*a} that
$$
\mmm(\widetilde{A}(x,r_x,30) )\geq \frac{\hhh}{2-\hhh}\mmm(B(x,r_{x}/4))
\geq \frac{\hhh}{2-\hhh} \cdot \frac{C}{n_0^{2}}\mmm( {E})\cdot 2^{-2{n_0} },
$$
and the result follows.
 \end{proof}

 We are now ready to combine the previous arguments to prove Theorem \ref{mainth4}.

 \medskip

%
%
%

As in Lemma \ref{lem_minor}, select points $(y_{j})$, $j=1,...,M_{n_0} ^{*}$, 
 such that the balls    $B(y_{j},2^{-5n_0}) \subset B(x^*, 4\cdot 2^{-{n_0}})$ are pairwise disjoint and  $y_{j}\in B(x^*, 2^{-{n_0}}/100)\cap \widetilde{E} _{n_0}$.  
 
By construction,  the annuli
${A}(y_i,r_{y_i},30) $
and the sets  $ \widetilde{A}(y_i,r_{y_i},30) $ satisfy the assumptions of Lemmas \ref{lemmcor}  and \ref{lemma55}.  Also, for any $x'\in \widetilde{A}(y_i,r_{y_i},30) \cap\widetilde{A}(y_j,r_{y_j},30) $,  one has  $\rho_{x'} >10\cdot 2^{-{n_0} -1}$. Hence, for any $r<10\cdot 2^{-{n_0} -1}$, by \eqref{*th4*3*a} one necessarily also has  $
\mmm(B(x',r))\leq r^{0.88}$.

Then, as stated by Lemma  \ref{lemmcor}, the diameter of each of the (at most) two connex   parts of ${A}(y_i,r_{y_i},30) \cap {A}(y_j,r_{y_j},30) $ is smaller than $2^{-13 n_0}<24\cdot 2^{-13.5 n_0}$.  These connex parts are included in an annulus ${A}(y_i,r_{y_i},30)$, so it is a very thin region (the width of the annulus is   less than  $2^{-30n_0}$).
Hence, the intersection of $E$ with the union of the two connex parts    can be covered by at most $M^{**}$ balls of the form $B(z_{\ell },2^{-30n_0})$, where 
\begin{align}
\label{*th4*12*b}
& z_{\ell }\in \widetilde{A}(y_i,r_{y_i},30) \cap\widetilde{A}(y_j,r_{y_j},30),\\
\label{*th4*12*a}
& M^{**} \leq   C^{**}  \frac{2^{-13 n_0}}{ 2^{-30n_0}} = C^{**}2^{17 n_0}, 
\end{align}
and  $C^{**}$  is a positive   constant  depending only on the dimension.

 Also, by \eqref{*Atdef} and \eqref{*th4*12*b},  
 one sees that $\rrr_{z_{\ell }}>2^{-n_0-1}\cdot 10>2^{-30 n_0}$ for all $z_{\ell}$s.
Hence, \eqref{*th4*3*a} yields 
\begin{equation}\label{*th4*12*c}
\mmm(B(z_{\ell },2^{-30 n_0}))\leq 2^{-30n_0\cdot 0.88}=2^{-26.4 n_0}.
\end{equation}
This together with  \eqref{*th4*12*a} imply that 
\begin{align}\nonumber
\mmm\big (\widetilde{A}(y_i,r_{y_i},30) \cap\widetilde{A}(y_j,r_{y_j},30)\big )& \leq M^{**}\cdot 2^{-26.4 n_0}
<C^{**} 2^{17 n_0}\cdot 2^{-26.4 n_0}\\
\label{*th4*13*a}
&=C^{**}2^{- 9.4 n_0}.
\end{align}

In addition,  \eqref{*th4*10*b} gives
\begin{equation}\label{*th4*13*b}
\mmm \big (\widetilde{A}(y_i,r_{y_i},30)\big )\cdot \mmm \big(\widetilde{A}(y_j,r_{y_j},30)\big )\geq
\Big (\frac{\hhh}{2-\hhh}\Big )^{2} \cdot \frac{C^2}{n_0^{4}} \mmm^{2}( {E})\cdot 2^{-4n_0}.
\end{equation}
Hence by \eqref{*th4*13*a} for large $n_0$,  
\begin{equation}\label{*th4*13*c}
\mmm(\widetilde{A}(y_i,r_{y_i},30) \cap \widetilde{A}(y_j,r_{y_j},30) )<  2^{-5.3 n_0 } \mmm(\widetilde{A}(y_i,r_{y_i},30) )\cdot \mmm(\widetilde{A}(y_j,r_{y_j},30) ).
\end{equation}

Finally,    all sets $\widetilde{A}(y_i,r_{y_i},30) $ are included in $B(x^*,4\cdot 2^{-n_0})$, and their cardinality by  Lemma  \ref{lem_minor} is greater than $2^{2.1 n_0}$. So, one has 
\begin{align*} 
\mmm(B(x^*,4\cdot 2^{-n_0}))& \geq \sum_{i=1}^{2^{2.1n_0}} \mmm(\widetilde{A}(y_i,r_{y_i},30) )\\
&  -  \sum_{i,j=1: \, i\neq j }^{2^{2.1n_0}}  \mmm(\widetilde{A}(y_i,r_{y_i},30) \cap \widetilde{A}(y_i,r_{y_i},30)) \\
& \geq
\sum_{i=1}^{2^{2.1n_0}}\mmm(\widetilde{A}(y_i,r_{y_i},30) )\Big (1-
\sum_{j=1}^{2^{2.1n_0}}2^{-5.3 n_0}\mmm( \widetilde{A}(y_j,r_{y_j},30) ) \Big )\\
& \geq  \sum_{i=1}^{2^{2.1n_0}}\mmm(\widetilde{A}(y_i,r_{y_i},30) ) (1-
2^{2.1n_0}\cdot 2^{-5.3n_0}),
\end{align*}
where at the last step we simply used that   $\mmm( \widetilde{A}(x_j,r_{x_j},30) )\leq 1$. Then,   \eqref{*th4*10*b} yields that 
when $n_{0}\geq N_{0}$ is sufficiently large,
\begin{align*} 
\mmm(B(x^*,4\cdot 2^{-n_0}))&  \geq \sum_{i=1}^{2^{2.1n_0}}\mmm(\widetilde{A}(y_i,r_{y_i},30) ) (1-
2^{2.1n_0}\cdot 2^{-5.3n_0})\\
& \geq \frac{1}{2}\sum_{i=1}^{2^{2.1n_0}}\mmm(\widetilde{A}(y_i,r_{y_i},30) )  \\
& \geq 2^{2.1 n_0} \frac{\hhh}{2-\hhh}\cdot \frac{C}{n_0^{2}}
\mmm( {E})2^{-2n_0},
\end{align*}
which is greater than 1 as soon as  $N_{0}$ (hence $n_0 $) is chosen sufficiently large. This contradicts the fact that
$\mmm(\R^{2})=1$, 
and completes the proof.    

\begin{remark}
It would be natural to check if  a version of Theorem \ref{mainth4} holds in which the constant
$0.89 $ can be pushed down to a value  closer to zero, maybe at the price that 
$\ddd=30$ is replaced by a larger number.
We point out that the estimates \eqref{*th4*12*a} and \eqref{*th4*13*a}
show that in our arguments the order of estimate of $D^{*}_{corr,n}$ in 
Lemma \ref{lemmcor} is crucial. In \eqref{*th4*13*a} in the end of the inequality, we need a (sufficiently large) negative power of $2$. 
If $0.89$ is replaced by $0.56=17/30$, using \eqref{*th4*12*a} and
\eqref{*th4*12*c} one can see that in the crucial estimate \eqref{*th4*13*a}, the power of $2$ will become non-negative.
Since the order of the estimate in Lemma \ref{lemmcor} is best possible, then one cannot improve significantly   \eqref{*th4*12*a} and hence the other estimates 
depending on it (by tighter estimates the exponent $30-13=17$ in \eqref{*th4*12*a} can be replaced by $30-13.5+\eee=16.5+\eee$).
\end{remark}

\subsection{Proof of Lemma \ref{lemmcor}}
\label{sseclemcor}

Without limiting generality we can suppose $r_{1}\geq r_{2}$ and can choose a  coordinate system in which $z_{1}=(0,y_{1})$, $z_{2}=(0,y_{2})$ and $y_{1}=r_{1}$. 
See Figure \ref{mcafig2} for an illustration (the figures are of course distorted, since   $2^{-30 n}$ is much   smaller than $2^{-n}$, so on a correct figure one of them cannot be shown, due to pixel size limitations).

In the proof,  when constants are said to depend  on the dimension 2 only, they do not depend  on other parameters - similar constants exist in higher dimensions as well.

 With this notation, \eqref{*eqlemcorb} means that
\begin{equation}\label{*eqlemcoro}
2^{-5 n}\leq |{y_{1}-y_{2}}|\leq 2^{-n}/30 .
\end{equation}
Of course the last inequality holds for sufficiently large  $n$s.
\begin{equation}\label{*supstrip}
\text{If $ {A}(z_1,r_{1},30)  \cap {A}(z_2,r_{2},30) $ is included }
\end{equation}
$$\text{ in the strip $[-8\cdot 2^{-13.5n},8\cdot 2^{-13.5n}]\xx \R$,}$$
 then its diameter is less than 
$24\cdot 2^{-13.5 n}$.

Assume now  that  $ {A}(z_1,r_{1},30)  \cap {A}(z_2,r_{2},30) $ is not included  in the strip $[-8\cdot2^{-13.5n},8\cdot2^{-13.5n}]\xx \R$.

We consider one of the two connex parts of $ {A}(z_1,r_{1},30)  \cap {A}(z_2,r_{2},30) $, the one located in the right half-plane. We denote its closure by $\cac_R$. 
See Figures \ref{mcafig2} and \ref{mcafig3}. 
The other part is symmetric and a similar estimate is valid for it.

Assume that   $r_{1}$ is    fixed and   $r\in [r_2-r_2^{30},r_2]$. 
Denote by  $M^1 (r) =(x^1(r),y^1(r) )$ the intersection of the circles   of radii $r_1-r_1^{30} $ and $r$, centered respectively at $z_1$ and $z_2$. 
We put $A=M^1(r_2)=(x_A,y_A)$, and $B=M^1(r_2-r_2^{30})=(x_B,y_B)$.
Observe that the points $M^1(r) $, $r\in [r_2-r_2^{30},r_2]$ are located on the arc with endpoints $A$ and $B$. On Figure \ref{mcafig2} the point $A$
is the point of the closed region $\cac_{R}$ with the largest abscissa.
Then $x_A=x^1(r_{2})>8\cdot2^{-13.5n}$.
However, as the left half of Figure \ref{mcafig3} illustrates,  for 
$r_{1}\approx r_{2}$ it may happen that not $A=M^1(r_2)=(x_A,y_A)$
is the point of $\cac_{R}$ with the largest abscissa.
On the left half of Figure \ref{mcafig3} this point is $C$.
In the sequel we suppose that $x_A=x^1(r_{2})>8\cdot2^{-13.5n}$.
The other cases can be treated analogously, the main point is that,
on the boundary of $\cac_R$, there is at least one point with abscissa 
larger than $8\cdot2^{-13.5n}$.

The abscissa $x^1(r)$   satisfies the implicit equation
$$\mathbf{F}_1 (r, {x}^1(r)):=y_{2}-y_{1}+\sqrt{r^{2}-( {x}^1(r))^2}-\sqrt{(r_{1} -r_1^{30})^{2}- ({x}^1(r))^2}=0. $$
Observe that by our assumption,  the intersection point lies  in the first quadrant $ {x}^1(r)>0$.
By implicit differentiation and  after simplification,
\begin{align}
\nonumber
({x} ^1)'(r) & =-\frac{\dd_{1}\mathbf{F}_1 (r, {x}^1 (r))}{\dd_{2}\mathbf{F}_1(r, {x}^1 (r))} =\frac{r\sqrt{(r_{1} -r_1^{30})^{2}- ({x}^1(r))^2}}{ x^1(r) (y_{1}-y_{2})}.
\label{*bfxde}
\end{align}

%
%
%

From the last equation, one deduces  by \eqref{*eqlemcoro}, $r\in [r_{2}-r_{2}^{30},r_{2}]$,  $2^{-n-1}\leq r_{1},r_{2}\leq 2^{-n}$   that
$$ |( {x}^1)'(r) {x}^1(r)| \leq \frac{r\sqrt{(r_{1} -r_1^{30})^{2}- ({x}^1(r))^2} }{ (y_{1}-y_{2})} \leq  \frac{2^{-2n}}{2^{-5n} } \leq 2^{3n}. $$
By integration, $ | ({x}^1(r))^2 -  ({x}^1(r_2))^2| \leq 2\cdot  2^{3n} |r-r_2| \leq 2^{3n+1 } 2^{-30n} \leq  2^{-27n+1}$, and 
hence
$$
  | {x}^1 (r) -  {x}^1(r_2)| \leq  \frac{ 2^{-27n+1}}{ | {x}^1 (r) + {x}^1(r_2)| }.
$$
  
Finally using that $x^1(r_{2})=x_{A}\geq8\cdot 2^{-13.5n}$
the previous equation gives that for any $r\in [r_2-r_2^{30},r_2]$
\begin{equation}
\label{eq12}
  | {x}^1 (r) -  {x}^1(r_2)| \leq 2^{-27n+1} 2^{13.5n-3} = 2^{-13.5n-2}  \text{, so  }|x_A-x_B|< 2^{-13.5n-2}
  \end{equation}
 and $x_B>8\cdot 2^{-13.5n}-2^{-13.5n-2} = 7.75 \cdot 2^{-13.5 n-1}$.
  
 \begin{figure}
\centering{
\includegraphics[width=1\textwidth]{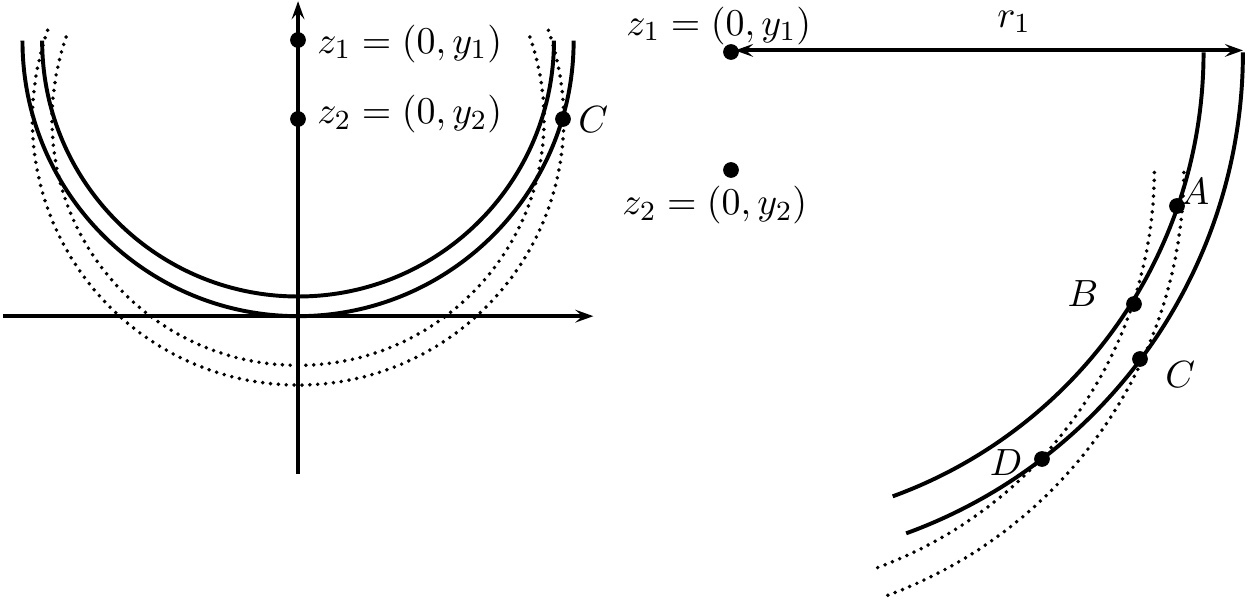}}
\caption{Special position of annuli and calculating the diameter of their intersection.}
\label{mcafig3}
\end{figure}

Assume now that   $r_{2}$ is    fixed and   $r\in [r_1-r_1^{30},r_1]$. 
Denote by  $M^2(r) =(x^2(r),y^2(r) )$ the intersection of the circles   of radii $r$ and $r_2-r_2^{30} $, centered respectively at $z_1$ and $z_2$. 
We put $D=M^2 (r_1)=(x_D,y_D)$.
Observe that the points $M_2 (r) $, $r\in [r_1-r_1^{30},r_1]$ are located on the arc with endpoints $B$ and $D$. 

Using the implicit equation $$\mathbf{F}_2 (r, {x}^2(r)):=y_{2}-y_{1}+\sqrt{(r_2-r_2^{30})^{2}- ({x}^2(r))^2}-\sqrt{r^{2}- ({x}^2(r))^2}=0,$$ the same considerations show that
 for any $r\in [r_1-r_1^{30},r_1]$ when $n\geq N_{corr}$ is sufficiently large
\begin{equation}
\label{eq122}
  | {x}^2 (r) -  {x}^2(r_2)| \leq 2^{-27n+1} 2^{13.5 n}/7.75 = 2^{-13.5n+1}/7.75,   
  \end{equation}
$$\text{ so }|x_B-x_D|\leq 2^{-13.5n+1}/7.75,$$
and $x_D>7.75\cdot 2^{-13.5n}- 2^{-13.5n+1}/7.75> 7\cdot  2^{-13.5n}$.

One can also consider the curve  $M^3(r) =(x^3(r),y^3(r) )$ connecting the points $D$
and $C$ on the boundary of $\cac_R$ and the curve  $M^4(r) =(x^4(r),y^4(r) )$ connecting the points $C$
and $A$ on the boundary of $\cac_R$. Estimates analogous to equations 
\eqref{eq12} and \eqref{eq122} are valid for these curves as well with
constants $7.75$ and $7.5$ decreased to $7$ and $6$. 
Since $\cac_{R}$ is compact, we can choose points $P=(x_{P},y_{P})$ and $Q=(x_{Q},y_{Q})$
on the boundary of $\cac_R$ such that the distance between $P$ and $Q$
equals the diameter of $\cac_R$.
Since these points can be connected by no more than three of the above mentioned arc segments we deduce that
\begin{equation}\label{*xpqest}
|x_P-x_Q|\leq 2^{-13.5 n}.
\end{equation}

The analogous estimate 
\begin{equation}\label{*ypqest}
|y_P-y_Q|\leq 2^{-13.5 n}
\end{equation}
is also true. In fact in this case the calculations are even simpler,
and most of the details are left to the reader.
We mention here only that, for example, for the function $y^1 (r)$
 one has a much simpler implicit equation
\begin{equation}\label{*Fey}
\mathbf{F}_3(r,y^1(r)):=r_1^{2}-r^{2}-(y^1(r)-y_{1})^{2}+(y^1(r)-y_{2})^{2}=0
\end{equation}
and by implicit differentiation 
\begin{equation}\label{*bfyde}
({y}^1)'(r)=
\frac{r}{y_{1}-y_{2}}.
\end{equation}

From this and \eqref{*supstrip} one concludes that the diameter of 
$\cac_{R}$
is less than $24\cdot 2^{-13.5n}$.

This concludes the proof, since the symmetric part (i.e. when $x_A<0$) is treated similarly.

\section{Methods related to Falconer's distance set problem}\label{*secdist}

The study of thin annuli and spherical averages is an important issue in many dimension-related problems, including Kakeya-type problems and Falconer's distance set conjecture. Recall that the distance set $D(E)$ of the set $E\sse \R^{d}$ is defined by
$$D(E):=\{ ||x-y||_2 : \  x,y\in E \}.$$
Falconer's distance set problem is about finding bounds of Hausdorff measure and dimension of $D(E)$ in terms of those of $E$.  Examples of Falconer show
that if $s\leq d/2$ then there are sets $E\sse \R^d$ such that $\dim_{H} E=s$
and $D(E)$ is of zero (one-dimensional) Lebesgue measure. It is conjectured that $E$ has positive Lebesgue measure as soon as $\dim_H E> d/2$. 
In one of the most recent results in the plane ($d=2$)  \cite{GuthFad},    it is proved that if $E$ is compact and $\dim_H E>5/4$ then
$D(E)$ has positive Lebesgue measure. For further details about Falconer's distance set problem we also refer to
\cite{Fadist} and Chapters 4, 15 and 16 of \cite{MattilaFou}.

Using standard arguments from \cite{Fadist,MattilaFou}, which is a different approach from the one developed earlier in this paper,  we can prove Proposition \ref {*anonym}.
 
 \begin{proof}
Let $t > 1/2$, $0<\eta<1$ and let $\mu$ be a finite $t$-regular measure on $\R^2$ with compact support satisfying \eqref{*treg}. 
Since we work in $\R^{2}$ the local dimension of $\mmm$ cannot exceed $2$,
so $2\geq t>1/2$.

Since $ \mu$ is $t$-regular, for every $s<t$, the   $s$-energy  of  $\mu$
 defined by $$
I_s (\mu)=\iint_{(\R^2)^2}
||x-y||_{2}^{-s}d\mmm(x)d\mmm(y)$$
is finite. Recall  also that  
\begin{equation}
\label{Ismu-fourier}
I_s (\mu)=\int_{\R^2} |\xi |^{s-2} |\hat\mu(\xi)|^2 d\xi ,
\end{equation}
 and for   every compactly supported function $\chi$ on $\R^d$, one has
\begin{equation}
\label{eq-transform}
\iint_{(\R^2)^2} 
\chi(x-y) d\mmm(x)d\mmm(y) =\int_{\R^2}
|\hat\chi(\xi) | \cdot |\hat\mu(\xi)|^2 d\xi,
\end{equation}
  where $\hat\mu$ and $\hat \chi$ are the Fourier transform of $\mu$ and $\chi$.

 Let 
\begin{equation}\label{*p12}
s:=\frac{1}{2}(t+\frac{1}{2})\text{ and }\ddd:=4=2(t-1/2)/(s-1/2)>(t-1/2)/(s-1/2).
\end{equation}
Oue estimations on $t$ imply $1/2<s\leq 5/4.$

Set  $\chi(x) = {\bf 1\!\!\!1} _{[r-r^\delta,r]}(|x|)$.
By Lemma 2.1 of \cite{Fadist}, 
\begin{equation}
\label{eqfalc}
|\hat\chi(\xi)| \leq C  r^{1/2} |\xi|^{-1/2} \min (r^{\delta}, |\xi|^{-1}).
\end{equation}

Following Falconer's argument (Theorem 2.2 of \cite{Fadist}) (see also \cite[Lemma 12.13]{MattilaFou}), one gets  by  \eqref{Ismu-fourier}, \eqref{eq-transform} and \eqref{eqfalc} that, 
keeping in mind that  $1/2<s\leq 5/4<3/2$
\begin{align*}
\iint_{(\R^2)^2}  \mmm(A(x,r,\ddd))d\mmm(x)& = \int_{\R^2}
|\hat\chi(\xi) | |\hat\mu(\xi)|^2 d\xi\\
&\leq C r^{1/2} \Big ( r^{\ddd} \int_{|\xi|\leq r^{-\delta} }
|\xi|^{-1/2}  |\hat\mu(\xi)|^2d\xi \\
& \hspace{5mm} +   \int_{|\xi |> r^{-\delta} }
|\xi|^{-3/2}   |\hat\mu(\xi)|^2 d\xi \Big ) \\ 
&\leq C r^{1/2} \Big ( r^{\ddd} \int_{|\xi|\leq r^{-\delta} }
|\xi|^{-1/2+(s-3/2)} |\xi|^{3/2-s}  |\hat\mu(\xi)|^2d\xi \\
& \hspace{5mm} +   \int_{|\xi |> r^{-\delta} }
|\xi|^{-3/2+s-1/2}|\xi|^{1/2-s}   |\hat\mu(\xi)|^2 d\xi \Big ) \\
& \leq C r^{1/2+\ddd(s-1/2)}I_{s}(\mu),
\end{align*}
for some constant $C>0$ that depends on $\mu$ and might change from line to line. 
Consequently, by Chebyshev's inequality, and the lower bound in \eqref{*treg}, we have
\begin{align}\label{*p1}
\mmm(\{ x:\mmm(A(x,r,\ddd))\geq \eta\cdot \mmm(B(x,r)) \})
& \leq
\mmm(\{ x:\mmm(A(x,r,\ddd)\geq \eta\cdot cr^{t} \})\\
\nonumber
& \leq C_{\mmm}' r^{1/2+\ddd(s-1/2)}I_{s}(\mmm)/ (\eta r^t)\\
\label{eq-ismu}
&=
C_{\mmm}'\eta^{-1}I_{s}(\mmm) r^{t-1/2},
\end{align}
where at the last equality we used \eqref{*p12} and $C_{\mmm}'$ is a suitable constant not depending on $r$.
Hence the right-hand side tends to zero as $r\searrow 0$.
This shows \eqref{*eqanonym}
 with an additional decay rate faster than  $r^{t-1/2}$, and thus completes the proof.
\end{proof}

 However  the above convergence in measure of Proposition \ref{*anonym} is not fast enough to hope to recover Theorems \ref{mainth1} to \ref{mainth4}, at least for the moment.
 
 Let us justify this claim.   Consider a measure $\mu$  supported on $[0,1]^2$ (to ease the argument)  such that the assumptions of Proposition \ref{*anonym} hold.  We would like to apply \eqref{*eqanonym}
 to deduce some estimate for the measure of $E_{\mu}(\delta,\eta)$.
 
 For this, consider equi-distributed points  $(r_{k,m})_{m=0,...,2^{3k}}$  in the interval $[r_{k,0}=2^{-k-1}, r_{k,2^{3k}}=2^{-k}]$. The distance between two consecutive $r_{k,m}$ and $r_{k,m+1}$  is $2^{-4k-1}$. If $x$ is such that $P_{\mu }(x,r,4,\eta)$ holds true for  $r\in[2^{-k-1},2^{-k}]$, then   $P_{\mu }(x,r_{k,m},4,\eta/2)$   holds for some $m$.  So, the set of points $\{x:P_{\mu }(x,r,4,\eta)$ holds true for some $r\in[2^{-k-1},2^{-k}]\}$  has $\mu$-measure less than
\begin{align*}
\sum_{k=0}^{2^{3k}} \mu(\{x:P_{\mu }(x,r_{k,m},4,\eta/2) \mbox{ holds true}\}  &  \leq \sum_{k=0}^{2^{3k}}C \eta^{-1}2^{-k(t-1/2)} \\
& \leq C 2^{ k(7/2-t)} 
 \end{align*}
 by \eqref{eq-ismu}. Unfortunately, 
 keeping in mind that $t\leq 2$ we have
 $\sum_{k}2^{ k(7/2-t)} =+\infty$ and the  Borel--Cantelli lemma cannot be applied (by far!) to prove Theorem \ref{mainth1} or Theorem \ref{mainth4}.
 
 Trying to optimize the choice of $s$ or $\delta$ (instead of 4) does not help either, using similar arguments.

\section*{Acknowledgments}

The authors thank Beno\^it Saussol for asking the question treated in this paper, Jean-Ren\'e Chazottes for interesting discussions and relevant references, 
Marius Urba\'nski for informing us about  \cite{Urbanski-annuli}  and   an anonymous reader   for turning our attention to the results and methods  developed for Falconer's distance problem.

\bibliographystyle{plain}
 
\bibliography{mca}

\end{document}